\documentclass[pdflatex,sn-mathphys-num]{sn-jnl}%

\usepackage{graphicx}%
\usepackage{multirow}%
\usepackage{amsmath,amssymb,amsfonts}%
\usepackage{amsthm}%
\usepackage{mathrsfs}%
\usepackage[title]{appendix}%
\usepackage{xcolor}%
\usepackage{textcomp}%
\usepackage{manyfoot}%
\usepackage{booktabs}%
\usepackage{listings}%
\usepackage{float}            % Better placement for figures/tables
\usepackage{placeins}         % \FloatBarrier
\usepackage{etoolbox}         % Hooks and tweaks
\usepackage{geometry}         % Optional – only if you override layout
\usepackage{csquotes}         % Robust quoting
\usepackage{adjustbox}        % For scaling/fit-to-width tables

\usepackage{hyperref}
\usepackage[nameinlink]{cleveref}

\usepackage{enumitem}
\usepackage{mathtools}
\usepackage[ruled,vlined]{algorithm2e}

\theoremstyle{thmstyleone}
\newtheorem{theorem}{Theorem}
\newtheorem{lemma}{Lemma}
\newtheorem{assumption}{Assumption}

\theoremstyle{thmstylethree}

\newtheorem{remark}{Remark}

\crefname{enumi}{Step}{Steps}
\crefname{enumii}{Step}{Steps}
\crefname{enumiii}{Step}{Steps}
\crefname{enumiv}{Step}{Steps}

\begin{document}

\title[DFO with Inexact ADMM]{Distributed Derivative-Free Optimization Using Inexact ADMM and Trust-Region Methods}

\author[1]{\fnm{Damilola} \sur{Fasiku}}\email{dfasiku@ncsu.edu}
\author*[1]{\fnm{Wentao} \sur{Tang}}\email{wtang23@ncsu.edu}

\affil[1]{\orgdiv{Department of Chemical and Biomolecular Engineering}, \orgname{North Carolina State University}, \orgaddress{\street{911 Partners Way}, \city{Raleigh}, \postcode{27695}, \state{North Carolina}, \country{USA}}}

\abstract{To reduce complexity and achieve scalable performance in high-dimensional black-box settings, we propose a distributed method for nonconvex derivative-free optimization (DFO) of continuous variables with an additively separable objective, subject to linear equality constraints. The approach is built upon the alternating direction method of multipliers (ADMM) as the distributed optimization framework. To handle general, potentially complicating linear equality constraints beyond the standard ADMM formulation, we employ a \textit{two-level ADMM structure}: an inner layer that performs sequential ADMM updates, and an outer layer that drives an introduced slack variable to zero via the method of multipliers. In addition, each subproblem is solved inexactly using a \textit{derivative-free trust-region solver}, ensuring suboptimality within a decreasing, theoretically controlled error tolerance. This inexactness is critical for both computational efficiency and practical applicability in black-box settings, where exact solutions are impractical or overly expensive. We establish theoretical convergence of the proposed approach to an approximate solution, building on analyses in related work, and demonstrate improved computational efficiency over monolithic DFO approaches on challenging high-dimensional benchmarks, as well as effective performance on a distributed learning problem.}


\maketitle

\section{Introduction}\label{sec:intro}

\subsection{Background}
Arguably, derivative information, when available, generally facilitates the convergence of iterative optimization methods, as this information is crucial to identifying effective search directions \cite{liang2024}. However, cases where derivative information is unavailable are not uncommon in scientific and engineering applications, including machine learning (e.g., hyperparameter tuning), process engineering (e.g., flowsheet design, computational fluid dynamics), and chemical/material computational design \cite{conn2009,boukouvala2014,bergstra2011,zaryab2022,terayama2021}. This can be due to the lack of an explicit mathematical model of the objective function, discontinuities, or non-differentiability at certain points, the unreliability of derivatives from noisy objective functions, and the unknown analytical form of the derivative \cite{larson2019}. Optimization in this sense is called \emph{derivative-free optimization (DFO)} or black-box optimization—an area of increasing research interest over the past two decades \cite{conn2009,boukouvala2014,larson2019,powell2003}.

The DFO methods in the literature can be broadly categorized into direct search and model-based methods \cite{larson2019,boukouvala2016}. Direct search methods iteratively explore the search space by comparing objective function values evaluated at strategically selected sample points. This category includes techniques such as Random Search, which evaluates the objective function at randomly chosen points within the domain, and the Nelder-Mead Simplex method, which adjusts a simplex of points toward lower function values \cite{rios2013}. Model-based DFO methods involve constructing approximate models of the objective function based on available data points. A popular optimization method in this category is Bayesian Optimization, a probabilistic approach that uses a Gaussian process to model the objective function and an acquisition function to select new evaluation points with a balance between exploration of uncertain regions with exploitation of promising ones \cite{frazier2018,shahriari2015}. The application of Bayesian optimization in engineering has been shown in molecular modeling and design, additive manufacturing, controller tuning, and bubble column technology \cite{wang2022,lu2020,kudva2022}, but its theoretical convergence largely remains an open question.

The \emph{trust-region (TR) methods}, when extended to derivative-free settings, form another class of model-based DFO methods. Here, surrogate models are constructed using function values within a defined trust region, and optimization steps are guided by these approximations, eliminating the need for derivative information \cite{liang2024,larson2019,eason2016,bajaj2018,conn2000}. The global convergence of trust-region (TR) methods in derivative-free settings to first- and second-order critical points for smooth black-box functions has been established in \cite{conn2009global}. This represents an improvement over the earlier TR-based method proposed in \cite{powell2003}, which demonstrated good practical performance but lacked theoretical convergence guarantees. In \cite{shashaani2018}, ASTRO-DF was introduced as a trust-region-based derivative-free optimization algorithm for stochastic settings, leveraging adaptive sampling to efficiently converge to first-order critical points. In \cite{ghanbari2017}, a trust-region-based derivative-free optimization (TR-DFO) method was compared to Bayesian optimization methods in machine learning applications, demonstrating superior performance in optimizing noisy functions, including the Area Under the Receiver Operating Characteristic Curve (AUC), with fewer evaluations. The trust-region method in \cite{conn2009global} was extended in \cite{garmanjani2016} to handle both smooth and nonsmooth functions, with worst-case complexity and convergence guarantees. However, the smoothing-based approach in \cite{garmanjani2016} is only effective when the source of the nonsmoothness is known, limiting its applicability in pure black-box settings. Inspired by bundle methods, \cite{Liuzzi2019} introduced a trust-region method suitable for black-box nonsmooth functions, with global convergence guarantees. The techniques and results from \cite{Liuzzi2019} will be discussed later in this work.

However, it should be noted that \emph{scalability of DFO methods to higher-dimensional problems} remains an open area of research. The direct use of existing methods often struggles in these settings due to the increase in computational and sample complexity, known as \enquote{the curse of dimensionality} \cite{rios2013,ling2017}, which frequently arises in applications such as drug discovery \cite{he2021} and materials design \cite{ling2017}. In the derivative‐free optimization literature, high‐dimensional problems are typically addressed via two strategies: (i) reducing dimensionality through variable selection or subspace embeddings, under the assumption of low intrinsic dimensionality \cite{qian2016,wang2016,song2022}; and (ii) exploiting structural properties of the objective function, such as additivity, partial separability, or overlapping group structures \cite{rolland2018,porcelli2022}. In the spirit of the second strategy above, this paper shows how an ADMM-based \emph{distributed optimization} framework can be employed to solve high-dimensional DFO problems with structured objectives by decomposing the monolithic problem into manageable subproblems. 

Historically, the list of decomposition strategies includes Dantzig-Wolfe decomposition \cite{dantzig1960}, which is suitable for large-scale linear programming problems, and Bender’s decomposition \cite{benders2005}, suitable for mixed-integer linear programming problems. These methods leverage duality in a structured way to decompose a problem into more manageable parts. However, dual decomposition methods require relatively strong assumptions, such as differentiability of the objective function, and often suffer from slow convergence when subgradient methods are used \cite{boyd2011}. The method of multipliers relaxes these assumptions while improving convergence stability by introducing a quadratic penalty term for constraint violations. However, this penalty term couples the primal variables, thereby preventing decomposition across subproblems \cite{boyd2011,zaryab2022}. ADMM combines the decomposability of dual decomposition with the convergence stability introduced by the quadratic regularization of the Method of Multipliers \cite{boyd2011}, making it highly suitable for distributed environments—particularly in pure black-box settings, which are characterized by the lack of derivative information and possible nonsmoothness; hence, its adoption in this work. Developed in the 1970s \cite{glowinski1975,gabay1976}, ADMM has been rediscovered in various forms, such as Douglas–Rachford splitting \cite{ryu2022} and Spingarn’s method of partial inverses \cite{spingarn1985}.

\subsection{Problem Description}

In this work, we consider a potentially high-dimensional optimization problem in derivative-free settings. We assume that this problem is separable, and linearly constrained as follows:

\begin{equation}
\begin{aligned}
\min_{(x_i)_{i=1,\ldots,N}, \, \bar{x}} & \sum_{i=1}^{N} f_i(x_i) + g(\bar{x}) \\
\text{s.t.} & \sum_{i=1}^N A_i x_i + B\bar{x} = b,
\end{aligned}
\label{eq:sep_opt_problem}
\end{equation}

\noindent where \( f_i : \mathbb{R}^{n_{x_i}} \rightarrow \mathbb{R} \) and \( g : \mathbb{R}^{n_{\bar{x}}} \rightarrow \mathbb{R} \) are both lower bounded functions. \( A_i \in \mathbb{R}^{m \times n_{x_i}} \), \( B \in \mathbb{R}^{m \times n_{\bar{x}}} \), and \( b \in \mathbb{R}^m \) are given data.  We assume that each \( f_i \) and \( g \) are black-box functions. It is important to state the equivalence of the problem structure in \eqref{eq:sep_opt_problem} to the following more conventional ADMM problem structure.

\begin{equation}
\begin{aligned}
\min_{x, \bar{x}} & \, f(x) + g(\bar{x}) \\
\text{s.t. } & \, Ax + B\bar{x} = b,
\end{aligned}
\label{eq:standard_admm_form}
\end{equation}

\noindent where \( f(x) = \sum_{i=1}^N f_i(x_i) \), \( A = [A_1 \dots A_N] \), and \( x = [x_1^\top, \dots, x_N^\top]^\top \).

The formulation in \eqref{eq:sep_opt_problem} is sufficiently general to capture unconstrained, partially separable high-dimensional DFO problems with arbitrary variable overlaps, by introducing local copies of the overlapping variables and enforcing consensus through the linear constraints. Another important instance within this framework is derivative-free consensus optimization, where multiple agents minimize the sum of local objective functions while enforcing agreement on a common decision variable without relying on gradients. This corresponds to the special case of \eqref{eq:sep_opt_problem} with \( g(\bar{x}) = 0 \), \( A_i = I \), and \( B = -I \), making \( \bar{x} \) the global consensus variable. Recent works  \cite{krishnamoorthy2023,krishnamoorthy2025} have investigated this setting by integrating Bayesian optimization with ADMM. Our general formulation extends beyond consensus to allow arbitrary linear couplings between subproblems.

\subsection{Relevant Literature}

In convex settings, the traditional ADMM has been both theoretically and practically shown to solve \eqref{eq:sep_opt_problem} with a linear convergence rate \cite{boyd2011,rodriguez2018}. More recently, attention has shifted to the analysis of ADMM in \textit{nonsmooth and nonconvex} settings \cite{hong2016,wang2019}. In particular, \cite{hong2016} established convergence of ADMM iterates to stationary points by employing the augmented Lagrangian as a Lyapunov function, provided that a sufficiently large penalty parameter is chosen. Subsequently, \cite{wang2019} derived similar convergence guarantees under weaker assumptions, thereby accommodating broader classes of nonconvex functions and constraint sets. While these results appear, in principle, extendable to the derivative-free setting of \eqref{eq:sep_opt_problem}, they rely on restrictive structural conditions on the constraint matrices—such as requiring \(B = I\) with \(A\) of full column rank \cite{hong2016}, or that \(\operatorname{Im}(A) \subseteq \operatorname{Im}(B)\) \cite{wang2019}. These conditions, however, are often violated in complex sharing and consensus problems.

In addition, most convergence results from these studies assumed that the subproblems are solved exactly, which may lead to excessive computational cost per ADMM iteration, especially in nonsmooth and nonconvex settings \cite{tang2022}. As a result, many studies have explored \textit{inexact ADMM,} in derivative-based settings where each subproblem is solved inexactly \cite{hager2019,bai2025,eckstein2017,xie2017,tang2022}. For the approximate update in each subproblem, the gradient is not required to be exactly 0 but should asymptotically converge to 0. The two main approaches to achieve this are: (i) assigning a diminishing and summable sequence of absolute errors, and (ii) using a sequence of relative errors that scale proportionally with other quantities in the algorithm~\cite{eckstein2017,xie2017}. Inexact ADMM method in convex settings are considered in \cite{eckstein2017}, while a distributed optimization algorithm for nonconvex constrained optimization, employing inexact ADMM, was developed in \cite{tang2022}. Recently, an inexact nonconvex ADMM method was introduced \cite{bai2025} under the same structural restrictions on \(A\) and \(B\) discussed above. This method achieves global and linear convergence when the subproblems satisfy restricted prox-regularity, and the penalty parameters are suitably chosen. 

 To address the structural restrictions on \( A \) and \( B \), \cite{sun2023} proposed a two-level distributed algorithm that introduces a slack variable constrained to zero. The problem is then relaxed by incorporating this constraint into the objective through its augmented Lagrangian, yielding a formulation where the inner ADMM layer solves the relaxed problem while an outer method-of-multipliers iteration progressively drives the slack variable to zero. This two-layered ADMM structure has also been employed in \cite{tang2022}, where a combination of inexact ADMM and a modified Anderson acceleration technique was used to improve the computational efficiency of the two-level framework in the presence of nonconvex constraints. The theoretical and practical performance of this framework has been demonstrated in these works. In this paper, we extend the two-layered framework to derivative-free settings and establish its convergence guarantees. Specifically, trust-region-based DFO methods \cite{conn2009global,Liuzzi2019}, as discussed earlier, are employed to solve the subproblems arising within the proposed ADMM structure.

\subsection{Summary of Contributions}

Our contributions are summarized below:

\textbf{1. A general algorithm for complex distributed structures.}
We propose a novel two-layer inexact ADMM algorithm to solve Problem~\eqref{eq:sep_opt_problem} in nonconvex and possibly nonsmooth black-box settings. Its primary advantage is the removal of restrictive assumptions on the coupling matrices \( A \) and \( B \), enabling its application to problems far beyond simple consensus, such as complex resource sharing and multi-level optimization. To the best of our knowledge, this is the first ADMM-based DFO method with theoretical guarantees that handles this level of generality.

\textbf{2. A practical inexact framework with a novel stationarity measure:}
To ensure computational tractability, we solve the ADMM subproblems inexactly. Our key innovation is to leverage the trust-region radius from DFO solvers as a practical and computable measure of the stationarity gap for each black-box subproblem minimization. This provides a rigorous mechanism to control the subproblem errors against dynamically decreasing tolerance sequences.

\textbf{3. Theoretical guarantees and numerical validation:}
We establish the global convergence of the overall algorithm to stationary points under mild assumptions. Furthermore, we demonstrate its effectiveness and computational advantage over monolithic solvers on high-dimensional, block-separable problems in derivative-free settings, including its application to gradient-free multi-agent optimization with general constraints.

The remainder of this work is organized as follows. Section~\ref{sec:prelim} reviews trust region-based DFO, ADMM, and the two-level ADMM framework. Section~\ref{sec:algo} details the proposed algorithm and its convergence theory. Numerical examples are presented in Section~\ref{sec:numerical} and conclusions are given in Section~\ref{sec:conclusion}.

\section{Preliminaries}\label{sec:prelim}

\subsection{Trust-Region-Based DFO: Smooth Case}

Trust-region methods are a major class of algorithms for derivative-free optimization, designed to solve the unconstrained problem
\begin{equation}\label{eq:generic_opt_problem}
\min_{x \in \mathbb{R}^n} h(x),
\end{equation}
where the objective function $h$ is accessible only through zeroth-order evaluations. At each iteration $q$, a local surrogate model $m_q(x)$ is constructed around the current iterate $x_q$, and the search for a candidate step is restricted to a trust region of radius $\Delta^q$. Within this region, the model must approximate $h$ with sufficient accuracy to ensure convergence. The radius is then adjusted based on how well the model predicts the actual reduction in $h$.

When $h$ is continuously differentiable, \cite{conn2009} proposed a trust-region framework that constructs fully linear models using previously evaluated function values. These models are updated iteratively to satisfy specific local accuracy requirements within the trust region. A model $m^q(x)$ at iteration $q$, defined for $x = x^q + s$ with $\|s\| \leq \Delta^q$, is said to be \textit{fully linear} if it satisfies
\begin{subequations}\label{eq:fullylinear}
\begin{equation}\label{eq:grad-approx}
\|\nabla h(x^q + s) - \nabla m^q(x^q + s)\| \leq \kappa_{eg}\Delta^q,
\end{equation}
\begin{equation}\label{eq:fun-approx}
|h(x^q + s) - m^q(x^q + s)| \leq \kappa_{ef}(\Delta^q)^2,
\end{equation}
\end{subequations}
where $\kappa_{eg}$ and $\kappa_{ef}$ are constants that bound the model’s gradient and function-value accuracy. The constants, \( \kappa_{eg} \) and \( \kappa_{ef} \), ensure approximations analogous to Taylor expansions, capturing both function values and local directional information. Despite the term “linear,” these models are often quadratic in structure \cite{conn2000}. A typical quadratic model takes the form:
\begin{equation}\label{eq:quadratic_model}
m^q(x) = h(x^q) + (g^q)^\top (x - x^q) + \tfrac{1}{2}(x - x^q)^\top H^q (x - x^q)
\end{equation}

\noindent where $g^q$ and $H^q$ are approximations of the gradient and Hessian of $h$ near $x^q$, constructed using interpolation or regression over a set of previously evaluated points.

To ensure attainability of a fully linear model, a model improvement algorithm is often used. This algorithm iteratively refines the set of sample points until the model meets the required accuracy conditions. For interpolation, this means ensuring the points are well-distributed (a concept called $\Lambda$-poisedness). For regression, it requires the basis matrix to be well-conditioned.

Now we present a trust-region DFO algorithm for smooth objectives, adapted from \cite{conn2009}. 
At each iteration, the algorithm maintains an incumbent model $m_{\mathrm{icb}}^q$ and trust-region radius $\Delta_{\mathrm{icb}}^q$, representing the most recently validated fully linear model and its associated radius. 
The model gradient $g^q := \nabla m^q(x^q)$ and Hessian $H^q := \nabla^2 m^q(x^q)$ serve as local approximations of the objective’s first- and second-order behavior. 
If the model gradient $\|g_{\mathrm{icb}}^q\|$ is sufficiently large, the current model is used directly; otherwise, a criticality check is performed, and the model is refined by reducing the radius until it satisfies the fully linear condition. 
A trial step is then computed by approximately minimizing the local model within the trust region, and the step is accepted or rejected based on the agreement between the actual and predicted reductions in the objective. 
The trust-region radius and model are then updated: successful steps lead to an increase in radius, while unsuccessful ones prompt model improvement and radius contraction. 
The complete steps are detailed in Algorithm~\ref{dfo_smooth_alg}.

We present below the assumptions for the global convergence of Algorithm~\ref{dfo_smooth_alg} in solving \eqref{eq:generic_opt_problem}.

\begin{assumption}\label{ass:smooth-lip-gradient}
The function $h$ is continuously differentiable and bounded below, with a gradient 
$\nabla h$ that is Lipschitz continuous on an open set containing the union of all trust regions 
generated by the algorithm.
\end{assumption}

\begin{assumption}\label{ass:hessian-bounded}
The model Hessians, defined in \eqref{eq:quadratic_model}, are uniformly bounded, that is, $\| H_{k,r,q} \| \le \kappa_{\mathrm{bhm}}$ for some constant $\kappa_{\mathrm{bhm}} > 0$.
\end{assumption}

Assumption~\ref{ass:smooth-lip-gradient} ensures the objective function is sufficiently well-behaved, guaranteeing that the Taylor-based models used in the trust-region framework are justified. Assumption~\ref{ass:hessian-bounded} ensures the curvature of the local models remains controlled, which is essential for the stability of the algorithm and the convergence analysis.

\begin{lemma}[Global convergence of Algorithm~\ref{dfo_smooth_alg}, \cite{conn2009}]
\label{lem:dfo_tr convergence}
Under Assumptions~\ref{ass:smooth-lip-gradient} and~\ref{ass:hessian-bounded}, the trust region radius satisfies
\begin{equation}\label{eq:tr_radius_zero}
\lim_{q \to \infty} \Delta^{q} = 0,
\end{equation}
and all limit points of the sequence $\{x^{q}\}$ are first-order stationary; hence
\begin{equation}\label{eq:gradient_zero}
\liminf_{q \to \infty} \| \nabla h(x^{q}) \| = 0.
\end{equation}
\end{lemma}

The convergence results in Lemma~\ref{lem:dfo_tr convergence} provide theoretical justification for using the trust-region radius as a surrogate measure of stationarity in smooth settings as this will be exploited later in Section~\ref{sec:algo}

\begin{algorithm}[H]
\caption{Trust-region DFO for smooth objective functions}
\label{dfo_smooth_alg}
\textbf{Input:} Initial point \(x^{0}\), incumbent model \(m^{0}_{\mathrm{icb}}\), initial radius \(\Delta^{0}_{\mathrm{icb}}\in(0,\Delta_{\max}]\).\\
\textbf{Constants:} \(0\le \eta_0 \le \eta_1 < 1,\ \varepsilon_c>0,\ \kappa_{\mathrm{fcd}}\in(0,1],\ \alpha\in(0,1),\ \zeta\in(0,1),\ \gamma_{\mathrm{inc}}>1,\ \mu>\tau>0.\)\\
Set iteration counter \(q=0\).

\medskip
\medskip
\medskip

\textbf{Step 1 (Criticality check).}
\quad\textit{Notation: } $g^{q}_{\mathrm{icb}} := \nabla m^{q}_{\mathrm{icb}}(x^{q})$, \ 
$\tilde g^{q} := \nabla \tilde m^{q}(x^{q})$.

\begin{enumerate}[label=\textbf{1.\arabic*},ref=1.\arabic*]
  \item \label{step:1.1}
    If $\|g^{q}_{\mathrm{icb}}\|>\varepsilon_c$: 
    set $m^{q}\gets m^{q}_{\mathrm{icb}}$, \ $\Delta^{q}\gets \Delta^{q}_{\mathrm{icb}}$.

  \item \label{step:1.2}
    Else ($\|g^{q}_{\mathrm{icb}}\|\le \varepsilon_c$): 
    attempt to certify $m^{q}_{\mathrm{icb}}$ fully linear on $B(x^{q};\Delta^{q}_{\mathrm{icb}})$.
    \begin{enumerate}[label=\textbf{1.\arabic{enumi}(\alph*)},ref=1.\arabic{enumi}(\alph*)]
      \item \label{step:1.2a}
        If not certifiable or $\Delta^{q}_{\mathrm{icb}}>\mu\|g^{q}_{\mathrm{icb}}\|$: 
        Perform the following model-improvement loop.\\ 
        Set $u\gets0$. \textbf{Repeat:}
        \begin{enumerate}
          \item $u\gets u+1$
          \item $\tilde{\Delta}^{q}\gets \alpha^{u-1}\Delta^{q}_{\mathrm{icb}}$
          \item Build $\tilde{m}^{q}$ fully linear on $B(x^{q};\tilde{\Delta}^{q})$ 
                and set $\tilde{g}^{q}=\nabla\tilde{m}^{q}(x^{q})$
        \end{enumerate}
        \textbf{until} $\tilde{\Delta}^{q}\le \mu\|\tilde{g}^{q}\|$. \\
        Set $m^{q}\gets \tilde{m}^{q}$, \ $\Delta^{q}\gets 
        \min\{\max(\tilde{\Delta}^{q},\,\tau\|\tilde{g}^{q}\|),\ \Delta^{q}_{\mathrm{icb}}\}$.

      \item \label{step:1.2b}
        Otherwise: $m^{q}\gets m^{q}_{\mathrm{icb}}$, \ 
        $\Delta^{q}\gets \Delta^{q}_{\mathrm{icb}}$.
    \end{enumerate}
\end{enumerate}

\medskip
\textbf{Step 2 (Step computation).} Compute \(s^{q}\) with \(x^{q}{+}s^{q}\in B(x^{q};\Delta^{q})\) such that
\[
m^{q}(x^{q})-m^{q}(x^{q}{+}s^{q})
\ \ge\
\frac{\kappa_{\mathrm{fcd}}}{2}\,\|g^{q}\|\,
\min\!\left\{\ \|g^{q}\|,\ \frac{\|g^{q}\|}{\|H^{q}\|},\ \Delta^{q}\ \right\}.
\]

\medskip
\textbf{Step 3 (Trial point acceptance).} Evaluate
\[
\rho^{q}=\frac{h(x^{q})-h(x^{q}{+}s^{q})}{\,m^{q}(x^{q})-m^{q}(x^{q}{+}s^{q})\,}.
\]
If \(\rho^{q}\ge \eta_1\) or \(\big(\rho^{q}\ge \eta_0\ \text{and $m^{q}$ fully linear on }B(x^{q};\Delta^{q})\big)\):
set \(x^{q+1}\gets x^{q}{+}s^{q}\) and update \(m^{q+1}_{\mathrm{icb}}\);
else set \(x^{q+1}\gets x^{q}\), \(m^{q+1}_{\mathrm{icb}}\gets m^{q}_{\mathrm{icb}}\).

\medskip

\textbf{Step 4 (Model improvement).} 
If $\rho^{q}<\eta_1$, attempt to certify that $m^{q}$ is fully linear on $B(x^{q};\Delta^{q})$. 
If certification fails, apply the same model-improvement loop as in Step 1 (Criticality check). 
Set $m^{q+1}_{\mathrm{icb}}$ to the resulting (possibly improved) model.

\medskip

\textbf{Step 5 (Trust-region radius update).}
\[
\Delta^{q+1}_{\mathrm{icb}}\in
\begin{cases}
[\Delta^{q},\ \min\{\gamma_{\mathrm{inc}}\Delta^{q},\,\Delta_{\max}\}], & \rho^{q}\ge \eta_1,\\
\{\zeta\,\Delta^{q}\}, & \rho^{q}< \eta_1 \ \text{and } m^{q}\ \text{fully linear},\\
\{\Delta^{q}\}, & \rho^{q}< \eta_1 \ \text{and } m^{q}\ \text{not certifiably fully linear}.
\end{cases}
\]
Set \(q\leftarrow q{+}1\) and return to Step 1.
\end{algorithm}

\vspace{5mm}

\subsection{Trust-Region-Based DFO: Nonsmooth Case}
For locally Lipschitz, nonsmooth objective functions $h(x)$, solving \eqref{eq:generic_opt_problem} 
with trust-region methods is more challenging, since smooth models cannot be directly employed. 
One approach is to apply smoothing techniques that approximate $h$, 
thereby enabling the use of smooth trust-region frameworks when the source of nonsmoothness is known \cite{garmanjani2016}. 
In this work, however, we adopt the alternative approach of \cite{Liuzzi2019}, which constructs max-linear surrogate models. 
This method combines ideas from bundle methods with a purely function-value-based strategy, 
making it applicable even in fully black-box settings. 
The model construction for this trust-region approach is described below.

At each iteration $q$, a nonsmooth trust-region model is built:

\begin{equation}\label{eq:nonsmooth_model}
m^q(s) = \max_{i \in \mathcal{I}^q} \left\{ f(x^q) + (g_i)^\top s - \tilde{\beta}_i^q \right\} + \tfrac{1}{2}s^\top H^q s.
\end{equation}

\noindent The index set $\mathcal{I}^q$ corresponds to a set of randomly generated, normalized directions $G^q = \{g_i : \|g_i\|=1\}$. The parameter $\tilde{\beta}_i^q$ for each direction $i \in \mathcal{I}^q$ is computed using a sample set $Y^q = \{y^j : j \in \mathcal{J}^q\}$, where the sample points satisfy $\|y^j - x^q\| \leq \gamma \Delta^q$ for a fixed constant $\gamma > 0$, as follows:
\[
\tilde{\beta}_i^q = \max_{j \in \mathcal{J}^q} \left( \max\left\{0,\, f(x^q) - f(y^j) + g_i^\top(y^j - x^q) + \delta \|y^j - x^q\|^2 \right\} \right),
\]
where $\delta > 0$ is a fixed algorithm parameter. A small $\tilde{\beta}_i^q$ indicates $g_i$ is a good approximation of a subgradient at a nearby point. The matrix $H^q$ is a symmetric Hessian approximation matrix built from quadratic interpolation or regression on the sample set $Y^q$. The iterate $x^q \in \mathbb{R}^n$ is the current best solution estimate, and $\Delta^q > 0$ is the current trust-region radius. The resulting trust-region subproblem is formulated as the following constrained quadratic program:
\begin{equation}\label{eq:reformulated_subproblem}
\min_{s, \alpha} \quad \tfrac{1}{2}s^\top H^q s + \alpha \quad \text{subject to} \quad (f(x^q) - \tilde{\beta}^q_i) + g_i^\top s \leq \alpha, \quad \forall i \in \mathcal{I}^q,\quad \|s\|^2 \leq (\Delta^q)^2.
\end{equation}

We present in Algorithm~\ref{dfo_nonsmooth_alg} a trust-region DFO algorithm adapted from \cite{Liuzzi2019} for minimizing nonsmooth locally Lipschitz functions. 
At each iteration, the algorithm constructs a max-linear surrogate model~\eqref{eq:nonsmooth_model} using sampled function values and randomly generated unit directions that approximate subgradients. 
The resulting trust-region subproblem, formulated as the convex quadratic program in~\eqref{eq:reformulated_subproblem}, is solved to obtain a trial step $s^q$. 
The Lagrange multipliers $\lambda_i^q$ from this subproblem define a convex combination of the active directions, yielding an aggregate vector 
$\tilde{g}^q = \sum_i \lambda_i^q g_i$ with $\sum_i \lambda_i^q = 1$, which serves as an approximate Clarke subgradient at $x^q$. 
A criticality test checks whether $\|\tilde{g}^q\|$ is sufficiently small relative to the trust-region radius; if so, the direction set is reset to avoid degeneracy. 
The step quality is then evaluated using a ratio test that compares actual and predicted reductions, and the trust-region radius is expanded or contracted accordingly. 
This adaptive process ensures convergence toward Clarke stationary points under some  assumptions which will be presented shortly.

\begin{algorithm}[H]
\caption{Trust-region DFO for nonsmooth subproblems}
\label{dfo_nonsmooth_alg}
\textbf{Initialization:} Choose \( x^{0} \in \mathbb{R}^n \), \( \Delta^{0} > 0 \), and constants \\
\( \bar{\eta}_1 > 0,\ \bar{\theta} > 0,\ \bar{\epsilon} > 0,\ \bar{\zeta}_1 < 1 \leq \bar{\zeta}_2,\ \bar{\zeta} > 0,\ \bar{p} > 0 \). \\
Set \( G^{0} = \emptyset \).

\textbf{For} \( q = 0, 1, 2, \ldots \):

\textbf{Step 1 (Model Construction):} \\
\quad Generate a unit vector \( g^{q} \). \\
\quad Let \( J^{q} \) index sample points \( y^{j,q} \) satisfying
\[
\|y^{j,q} - x^{q}\| \leq \bar{\zeta} \Delta^{q}, \quad \forall j \in J^{q}. 
\]
\quad Form direction set \( G^{q} = G^{q-1} \cup \{g^{q}\} \). \\
\quad Solve subproblem \eqref{eq:reformulated_subproblem} using \( G^{q} \) to obtain trial step \( s \).

\textbf{Step 2 (Subgradient Computation):} \\
\quad Let \( I^{q} \) be the index set corresponding to \( G^{q} \). \\
\quad Compute
\[
\tilde{g}^{q} = \sum_{i \in I^{q}} \lambda_i^{q} g_i,
\]
\quad where \( \lambda_i^{q} \geq 0 \) are the Lagrange multipliers associated with the constraints \( f(x^{q}) + g_i^\top s - \tilde{\beta}_i^{q} \leq \alpha \) in \eqref{eq:reformulated_subproblem}, satisfying \( \sum_{i \in I^{q}} \lambda_i^{q} = 1 \).

\textbf{Step 3 (Criticality Test):} \\
\quad If \( \|\tilde{g}^{q}\| < \bar{\epsilon} (\Delta^{q})^{1/2} \), then \\
\quad \quad Reset \( G^{q} = \{g^{q}\} \) \\
\quad \quad Re-solve subproblem \eqref{eq:reformulated_subproblem} to update \( s \).

\textbf{Step 4 (Step and Ratio Evaluation):} \\
\quad Set \( s^{q} = s \), and compute
\[
\rho^{q} = \frac{h(x^{q}) - h(x^{q} + s^{q})}{\bar{\theta} \|s^{q}\|^{\bar{p} + 1}}.
\]

\textbf{Step 5 (Step Acceptance and Radius Update):} \\
\quad \textbf{If} \( \rho^{q} \geq \bar{\eta}_1 \): \\
\quad \quad \( x^{q+1} = x^{q} + s^{q} \), \\
\quad \quad \( \Delta^{q+1} = \bar{\zeta}_2 \Delta^{q} \). \\
\quad \textbf{Else:} \\
\quad \quad \( x^{q+1} = x^{q} \), \\
\quad \quad \( \Delta^{q+1} = \bar{\zeta}_1 \Delta^{q} \).

\end{algorithm}

\vspace{1cm}

We present the convergence properties of Algorithm~\ref{dfo_nonsmooth_alg} for minimizing a locally Lipschitz continuous, potentially nonsmooth function $h$ in \eqref{eq:generic_opt_problem}, as taken from \cite{Liuzzi2019}. We first recall key concepts from nonsmooth analysis \cite{Clarke1998}. The \textit{Clarke generalized directional derivative} of $h$ at $x$ in the direction $d$ is
\begin{equation}
h^{\circ}(x; d) = \limsup_{y \to x,\, t \downarrow 0} \frac{h(y + t d) - h(y)}{t},
\label{eq:clarke_dir_derivative}
\end{equation}
and the \textit{Clarke generalized subdifferential} is the set
\begin{equation}
\partial_C h(x) = \{ \xi \in \mathbb{R}^n : h^{\circ}(x; d) \geq \xi^\top d, \ \forall d \in \mathbb{R}^n \}.
\label{eq:clarke_subdiff}
\end{equation}
A point $x^*$ is \textit{Clarke stationary} if $0 \in \partial_C h(x^*)$, which is equivalent to the condition that the generalized directional derivative is nonnegative in every direction:
\begin{equation}
h^{\circ}(x^*; d) \geq 0 \quad \text{for all } d \in \mathbb{R}^n.
\label{eq:clarke_stationarity}
\end{equation}

Compared to the classical gradient or convex subdifferential, the Clarke subdifferential provides a weaker notion of stationarity: every differentiable or convex stationary point is Clarke stationary, but not vice versa. In the nonsmooth optimization literature~\cite{Clarke1998,RockafellarWets1998}, however, convergence is typically analyzed with respect to Clarke stationarity, as it offers a broadly applicable and tractable first-order condition for locally Lipschitz functions. We formerly state the needed assumptions for the convergence of Algorithm~\ref{dfo_nonsmooth_alg} below.

\begin{assumption}\label{ass:local-lipschitz}
The function $h$ is locally Lipschitz continuous.
\end{assumption}
\begin{assumption}\label{ass:surrogate-growth}
There exist constants $\vartheta \in (0, 1/2)$, $m > 0$, and $M > 0$ such that the matrix $H^{q}$ in \eqref{eq:nonsmooth_model} satisfies
\[
\lambda_{\max}(H^{q}) \leq M (\Delta^{q})^{-\vartheta}, \quad -\lambda_{\min}(H^{q}) \leq m (\Delta^{q})^{-\vartheta}.
\]
\end{assumption}
\noindent Assumption~\ref{ass:surrogate-growth} is a key relaxation of standard trust-region theory, allowing the Hessian approximation to grow in a controlled manner to capture nonsmooth behavior. Under these assumptions, the convergence of Algorithm~\ref{dfo_nonsmooth_alg} is characterized by the following lemma, which comes from \cite{Liuzzi2019}.

\vspace{2mm}

\begin{lemma}\label{lem:nonsmooth_convergence}
Let Assumptions~\ref{ass:local-lipschitz} and~\ref{ass:surrogate-growth} hold. Then:
\begin{enumerate}
    \item $\lim_{q \to \infty} \Delta^{q} = 0$.
    \item For any limit point $x^{*}$ of $\{x^{q}\}$ and limit direction $s^{*}$ of $\{s^{q}/\|s^{q}\|\}$, $h^{\circ}(x^{*}; s^{*}) \geq 0$.
    \item If $\{g^{q}\}$ is asymptotically dense in the unit sphere, then every limit point $x^{*}$ is Clarke stationary: $0 \in \partial_C h(x^{*})$.
\end{enumerate}
\end{lemma}

Lemma~\ref{lem:nonsmooth_convergence} provides the theoretical justification for using the trust-region radius as a surrogate stationarity measure in nonsmooth settings, which will be exploited in Section~\ref{sec:algo}.

\subsection{Exact and Inexact ADMM}

The Lagrangian \(L\) and augmented Lagrangian \(L_{\beta}\) of the problem in \eqref{eq:standard_admm_form} are  given as:
\begin{equation}\label{eq:lagrangian}
L(x, \overline{x}; \lambda) = f(x) + g(\overline{x}) + \lambda^T (Ax + B\overline{x} - b)
\end{equation}
\vspace{-2mm}
\begin{equation}\label{eq:augmented_lagrangian}
L_{\beta}(x, \overline{x}; \lambda) = f(x) + g(\overline{x}) + \lambda^T (Ax + B\overline{x} - b) + \frac{\beta}{2} \|Ax + B\overline{x} - b\|^2
\end{equation}

\noindent where \(\lambda \in \mathbb{R}^m\) are the dual variables, and \(\beta > 0\) is a penalty parameter. The classical ADMM involves solving the minimization subproblems below exactly:
\begin{enumerate}
    \item $x$-update: 
    \begin{equation}\label{eq:x_update}
    x^{k+1} = \arg \min_x L_{\beta}(x, \overline{x}^k; \lambda^k);
    \end{equation}
    \item $\overline{x}$-update: 
    \begin{equation}\label{eq:xbar_update}
    \overline{x}^{k+1} = \arg \min_{\overline{x}} L_{\beta}(x^{k+1}, \overline{x}; \lambda^k);
    \end{equation}
    \item Dual update:
    \begin{equation}\label{eq:dual_update}
    \lambda^{k+1} = \lambda^k + \beta (A x^{k+1} + B \overline{x}^{k+1} - b).
    \end{equation}
\end{enumerate}
 \vspace{-2mm}
Inexact ADMM, on the other hand, allows these subproblems to be solved approximately to a specified error tolerance, which is crucial for reducing the computational cost per iteration, especially when the subproblems are themselves challenging optimization tasks. Rather than requiring exact minimizers, the $x$- and $\overline{x}$-updates are solved such that the distance to optimality—measured, for example, by the norm of the (sub)gradient of the augmented Lagrangian—is controlled by an error tolerance $\epsilon^k$ that diminishes across iterations. Two common criteria are used: an \textit{absolute error} criterion where $\epsilon^k$ forms a summable sequence (i.e., $\sum_{k=0}^\infty \epsilon^k < \infty$), or a \textit{relative error} criterion where $\epsilon^k$ is related to the algorithm's progress, such as the norm of the difference between successive iterates. These carefully chosen diminishing tolerances ensure that the approximation errors converge to zero asymptotically, guaranteeing that the sequence $\{(x^k, \overline{x}^k, \lambda^k)\}$ converges to a KKT point of the original problem. As this inexact approach is adopted in this work, more details will be provided in the next section.

\subsection{A Two-Level Distributed Algorithm}

To remove the coupling restrictions on \(A\) and \(B\) in standard nonconvex ADMM formulations, \cite{sun2023} proposed adding slack variables to the problem in~\eqref{eq:standard_admm_form}, yielding the equivalent problem below:

\begin{equation}
\begin{aligned}
\min_{(x, \overline{x}, z)} & \, f(x) + g(\overline{x}) \\
\text{s.t.} \quad & Ax + B\overline{x} + z = b, \quad z = 0.
\end{aligned}
\label{eq:slack_var_problem}
\end{equation}
\vspace{3mm}
\noindent Equation \eqref{eq:slack_var_problem} can be relaxed to yield:

\vspace{-3mm}
\begin{equation}
\begin{aligned}
\min_{(x, \overline{x}, z)} & \, f(x) + g(\overline{x}) + \lambda^T z + \frac{\beta}{2} \|z\|^2 \\
\text{s.t.} \quad & Ax + B\overline{x} + z = b.
\end{aligned}
\label{eq:relaxed_problem}
\end{equation}

\noindent The resulting augmented Lagrangian for the relaxed problem is:
\begin{equation}
\begin{aligned}
L(x, \overline{x}, z, y; \lambda, \rho, \beta)
&= f(x) + g(\overline{x}) + y^T (Ax + B\overline{x} + z - b)  \\
&\quad + \frac{\rho}{2}\|Ax + B\overline{x} + z - b\|^2 
   + \lambda^T z + \frac{\beta}{2}\|z\|^2.
\end{aligned}
\label{eq:two_level_lagrangian}
\end{equation}

\noindent Here, \( y \in \mathbb{R}^m \) and \( \lambda \in \mathbb{R}^m \) are dual variables, while \( \rho > 0 \) and \( \beta > 0 \) are penalty parameters.

In \cite{sun2023}, a two-layer structure is proposed, with an outer loop (indexed by \( k \)) and an inner loop (indexed by \( r \)). In the inner loop, the classical ADMM updates \( x \), \( \overline{x} \), \( z \), and \( y \) sequentially according to
\begin{equation}
\begin{aligned}  
x^{k, r+1} &= \arg\min_x L(x, \overline{x}^{k, r}, z^{k, r}, y^{k, r}; \lambda^k, \rho^k, \beta^k) \\
\overline{x}^{k, r+1} &= \arg\min_{\overline{x}} L(x^{k, r+1}, \overline{x}, z^{k, r}, y^{k, r}; \lambda^k, \rho^k, \beta^k) \\ 
z^{k, r+1} &= \arg\min_z L(x^{k, r+1}, \overline{x}^{k, r+1}, z, y^{k, r}; \lambda^k, \rho^k, \beta^k) \\
&\quad = -\frac{\rho^k}{\rho^k + \beta^k} \left( A x^{k, r+1} + B \overline{x}^{k, r+1} + \frac{y^{k, r}}{\rho^k}-b \right) - \frac{1}{\rho^k + \beta^k} \lambda^k \\
y^{k, r+1} &= y^{k, r} + \rho^k \left( A x^{k, r+1} + B \overline{x}^{k, r+1} + z^{k, r+1}-b \right)
\end{aligned}
\label{eq:inner_updates}
\end{equation}
while \( \lambda^k \), \( \beta^k \) and \( \lambda^k \) remain fixed. \cite{sun2023} showed that the inner iterations converge to the stationary point of the relaxed problem in \eqref{eq:relaxed_problem} under mild assumptions and \(\rho^k = 2\beta^k\). In the outer iterations, the dual variables and the corresponding penalty parameters are updated according to
\begin{equation}
\begin{aligned}
\lambda^{k+1} &= \Pi_{[\underline{\lambda},\overline{\lambda}]}\left(\lambda^k + \beta^k z^k\right), \\
\beta^{k+1} &= \begin{cases} 
\gamma \beta^k & \text{if } \|z^k\| > \omega \|z^{k-1}\| \\
\beta^k & \text{otherwise},
\end{cases}
\end{aligned}
\label{eq:outer_updates}
\end{equation}
where $\gamma > 1$ controls the penalty amplification rate when the slack variables do not sufficiently decrease, with $\omega \in (0,1)$ defining the convergence threshold. The projection operator $\Pi$ ensures that the dual variables $\lambda$ remain within the prescribed compact hypercube $[\underline{\lambda},\overline{\lambda}]$, which is a standard technique to guarantee the boundedness of the dual sequence and thus the stability of the augmented Lagrangian method. \noindent In summary, the two-level distributed algorithm solves the relaxed problem 
\eqref{eq:relaxed_problem} via ADMM updates in the inner layer, while the outer layer, 
implemented through the method of multipliers, drives the slack variables $z$ to zero. 
This decomposition eliminates the restrictive assumptions on $A$ and $B$ required in standard nonconvex ADMM analyses.

\vspace{-1mm}

\section{Proposed Algorithm}\label{sec:algo}
\vspace{-2mm}
\subsection{Overview}

We propose a two-layer inexact ADMM algorithm (Algorithm~\ref{proposed_alg}) for solving the derivative-free optimization problem in \eqref{eq:sep_opt_problem}. The algorithm's structure, depicted in Figure~\ref{fig:2_level_ADMM}, extends the framework from \cite{tang2022,sun2023} to derivative-free settings. The core idea is to decompose the problem into manageable subproblems for the blocks \(x_i\) and \(\bar{x}\), which are then solved inexactly using trust-region DFO methods, treating the objectives as black-box functions.

The subproblems for each variable block at inner iteration \(r\) of outer iteration \(k\) are defined by the following objectives:
\vspace{-3mm}
\begin{equation}\label{eq:x_subproblem_obj}
I_i(x_i) = f_i(x_i) 
    + \tfrac{\rho^k}{2}\Big\| A_i x_i 
      + \sum_{\substack{j=1 \\ j \ne i}}^{N} A_j x_j^{k,r} 
      + B \bar{x}^{k,r} + z^{k,r} + \tfrac{y^{k,r}}{\rho^k} \Big\|^2,
\end{equation}
\vspace{-3mm}
\begin{equation}\label{eq:xbar_subproblem_obj}
J(\bar{x}) = g(\bar{x}) + \frac{\rho^k}{2} \left\| A x^{k, r+1} + B \bar{x} + z^{k, r} + \frac{y^{k, r}}{\rho^k} \right\|^2.
\end{equation}

\begin{figure}[h]
    \centering
    \includegraphics[width=0.9\textwidth]{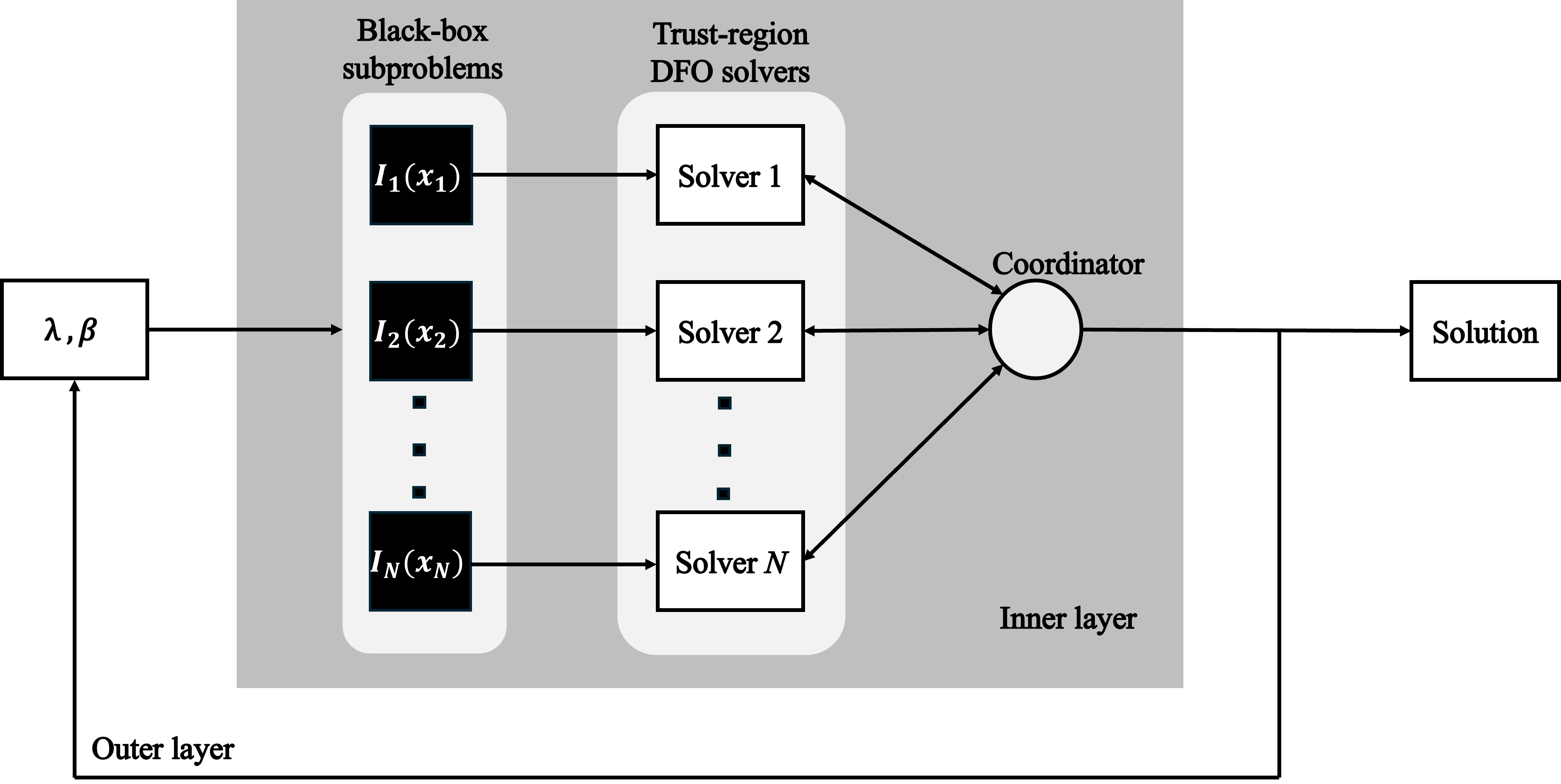}
    \caption{An illustration of the two-layer ADMM-DFO algorithm.}
    \label{fig:2_level_ADMM}
\end{figure}

\noindent We make the following standard assumptions to ensure the problem is well-posed and to facilitate convergence analysis.

\vspace{-3mm}

\begin{assumption}\label{ass:fi-and-g-lower-bounded}
Each function $f_i$ and $g$ is lower bounded.
\end{assumption}

\begin{assumption}\label{ass:closed_form_solution}
The function \(g\) is convex, and the minimization of the augmented Lagrangian with respect to \(\bar{x}\) in \eqref{eq:xbar_subproblem_obj} admits a unique closed-form solution, denoted by
\[
\bar{x}^{k,r+1} = G\!\left(B,\; A x^{k,r+1} + z^{k,r} + \tfrac{y^{k,r}}{\rho^k},\; \rho^k\right).
\]
\end{assumption}
\noindent Assumption~\ref{ass:closed_form_solution} is mild and holds for many practical problems, such as when \(g = 0\) or when \(g\) is a regularizer for which the proximal operator is known and \(B\) has full column rank.

\begin{algorithm}
\caption{ DFO-Based Two-Layer Inexact ADMM }
\label{proposed_alg}
\SetKwInOut{Input}{Input}
\SetKwInOut{Output}{Output}

% Reduce line spacing within the algorithm
\SetAlgoLined % Reduces line spacing in the algorithm

\Input{Bound of dual variables $[ \underline{\lambda}, \overline{\lambda} ]$, shrinking ratio $\omega \in [0, 1)$, amplifying ratio $\gamma > 1$, tolerances $\left\{ \epsilon_{1,2,3,4}^k \right\}_{k=1}^{\infty} \downarrow 0$, terminating tolerances $\epsilon_{1,2,3,4} > 0$}
\Output{Optimized variables $x^{k+1}, \bar{x}^{k+1}, z^{k+1}, y^{k+1}$}

\BlankLine
\textbf{Initialization:} Starting points $x^0, \bar{x}^0, z^0$, duals and bounds $\lambda^1 \in [ \underline{\lambda}, \overline{\lambda} ]$, penalty $\beta^1 > 0$\;
Outer iteration count $k \leftarrow 0$\;

\While{$\epsilon_4^k \geq \epsilon_4$ or stationarity criterion \eqref{eq:outer_optimality} is not met}{
    \textbf{Set:} Inner tolerances $\{\epsilon_4^{k,r}\}_{r=1}^\infty$; \\
    $\rho^k = 2\beta^k$\;
    Inner iteration count $r \leftarrow 0$\;
    \textbf{Initialization:} $x^{k,0}, \bar{x}^{k,0}, z^{k,0}, y^{k,0}$ satisfying $\lambda^k + \beta^k z^{k,0} + y^{k,0} = 0$\;

    \While {stopping criterion \eqref{eq:inner_stop_criteria} is not met}{
   \For{$i = 1, \ldots, N$}{
    Solve approximately (using Algorithm~\ref{dfo_smooth_alg} or Algorithm~\ref{dfo_nonsmooth_alg}):
    \[
    x_i^{k,r+1} \;\approx\; 
    \arg\min_{x_i}\; f_i(x_i) 
    + \tfrac{\rho^k}{2}\Big\| A_i x_i 
      + \sum_{\substack{j=1 \\ j \ne i}}^{N} A_j x_j^{k,r} 
      + B \bar{x}^{k,r} + z^{k,r} + \tfrac{y^{k,r}}{\rho^k} \Big\|^2
    \]
    to a trust–region radius tolerance of $\epsilon_{4}^{k,r}$.
}

        Update $\bar{x}^{k,r+1}$, $z^{k,r+1}$, and $y^{k,r+1}$ as:
\[
\begin{aligned}
\bar{x}^{k,r+1} &= G\!\left(B,\, A x^{k,r+1} + z^{k,r} + \tfrac{y^{k,r}}{\rho^k},\, \rho^k\right), \\[4pt]
z^{k,r+1} &= -\tfrac{\rho^k}{\rho^k + \beta^k}\!\left(Ax^{k,r+1} + B\bar{x}^{k,r+1} + \tfrac{y^{k,r}}{\rho^k}\right)
            - \tfrac{1}{\rho^k + \beta^k}\lambda^k, \\[4pt]
y^{k,r+1} &= y^{k,r} + \rho^k\!\left(Ax^{k,r+1} + B\bar{x}^{k,r+1} + z^{k,r+1}\right).
\end{aligned}
\]

        \( r \leftarrow r + 1 \)\;
    }
    \(\left( x^{k+1}, \bar{x}^{k+1}, z^{k+1}, y^{k+1} \right) \leftarrow \left( x^{k,r}, \bar{x}^{k,r}, z^{k,r}, y^{k,r} \right)\)\;

    Update the dual variable \(\lambda^{k+1}\) and the penalty parameter \(\beta^{k+1}\) according to:
    \[
    \lambda^{k+1} = \Pi_{\substack{[\lambda, \bar{\lambda}] \\ }} (\lambda^k + \beta^k z^k), \quad \beta^{k+1} = 
    \begin{cases} 
    \gamma \beta^k, & \|z^k\| > \omega \|z^{k-1}\| \\
    \beta^k, & \|z^k\| \leq \omega \|z^{k-1}\|
    \end{cases}
    \]

    \( k \leftarrow k + 1 \)\;
}
\end{algorithm}

\FloatBarrier

The main computational challenge lies in solving the black-box \(x_i\)-subproblems \eqref{eq:x_subproblem_obj}, as obtaining an exact minimizer is often impractical or prohibitively expensive. We therefore employ an \textit{inexact} minimization strategy. However, in derivative-free settings, standard first-order optimality measures, such as the gradient norm, are unavailable for tracking proximity to stationarity. This is where the convergence properties of trust-region DFO algorithms (Algorithms~\ref{dfo_smooth_alg} and~\ref{dfo_nonsmooth_alg}) become crucial: the trust-region radius serves as a practical and theoretically justified surrogate for quantifying the stationarity gap.

Now we introduce a three-level indexing \((k,r,q)\) for notational clarity, as this will be featured in the subsequent analysis. 
Here, \(k\) denotes the outer iteration, \(r\) the inner ADMM iteration, and \(q\) the trust-region iteration within a subproblem. 
In the inner loop of Algorithm~\ref{proposed_alg}, each \(x_i\)-subproblem is solved inexactly by a trust-region method, producing iterates \(x_i^{k,r,q}\). 
The trust-region solver terminates once the trust-region radius falls below a tolerance \(\epsilon_4^{k,r}\). 
Outer and inner iterations are governed by externally defined sequences of tolerances \(\{\epsilon_{1:4}^k\}\downarrow 0\) and \(\{\epsilon_{1:4}^{k,r}\}\), which together control deviations from optimality conditions and constraint violation. 
This framework ensures that each trust-region subproblem is solved only to the precision required at its stage of the algorithm, while the overall scheme converges globally.

\subsection{Convergence Analysis of Algorithm  \ref{proposed_alg}}\label{subsec:convergence}

\begin{lemma}[Smooth case: gradient norm--trust-region radius bound]
\label{lem:grad_norm_tr_bound}
For an inner subproblem $(k,r)$ of Algorithm~\ref{proposed_alg}, suppose that Algorithm~\ref{dfo_smooth_alg} produces trust-region iterates $\{x^{k,r,q}\}_q$ and terminates at index $q^\star$ with $x^{k,r+1}=x^{k,r,q^\star}$ and $\Delta^{k,r+1}=\Delta^{k,r,q^\star}$. Assume that the termination test is evaluated after the Criticality check (Step~\ref{step:1.1}--\ref{step:1.2}) of Algorithm~\ref{dfo_smooth_alg}, with the inner tolerance $\varepsilon^{k,r}_4$ sufficiently small so that $\|g_{\mathrm{icb}}^{k,r,q^\star}\|\le \varepsilon_c$ at termination (hence Step~\ref{step:1.2} of Algorithm~\ref{dfo_smooth_alg} is active). Let $\kappa_{eg}>0$ denote the full-linearity gradient constant,  $\tau\in(0,1)$ and $\mu>0$ the criticality constants from Algorithm~\ref{dfo_smooth_alg}, and define $d_4^k:=\nabla h(x^{k,r+1})$. Then
\begin{equation}
\label{eq:finalC}
\|d_4^k\|\ \le\ C_1(\tau,\mu,\kappa_{eg})\,\Delta^{k,r+1},
\qquad
C_1(\tau,\mu,\kappa_{eg})\ :=\ \max\!\Big\{\frac{1+\kappa_{eg}\mu}{\tau},\ \kappa_{eg}+\frac{1}{\mu}\Big\}.
\end{equation}

\end{lemma}

\begin{proof}
By the lemma assumption, termination occurs in Step~\ref{step:1.2} of Algorithm~\ref{dfo_smooth_alg}, where the model at $x^{k,r+1}$ is certifiably fully linear on $B(x^{k,r+1};\Delta^{k,r+1})$. 
This full linearity property is essential for establishing a bound on the true gradient norm in terms of the trust-region radius. 
Recall that whenever $m$ is fully linear on $B(x^{k,r+1};\Delta)$, the following inequality holds:
\begin{equation}
\label{eq:FL}
\|\nabla h(x^{k,r+1}) - \nabla m(x^{k,r+1})\|\ \le\ \kappa_{eg}\,\Delta.
\end{equation}

\textbf{Case 1 (Step~\ref{step:1.2a}).}
After the model-improvement loop, there exists a certified pair $(\tilde m^{k,r,q^\star},\tilde\Delta^{k,r,q^\star})$ that is fully linear on $B(x^{k,r,q^\star};\tilde\Delta^{k,r,q^\star})$ with
\begin{equation}
\label{eq:stop}
\tilde\Delta^{k,r,q^\star}\ \le\ \mu\,\|\tilde g^{k,r,q^\star}\|,
\qquad
\tilde g^{k,r,q^\star}:=\nabla\tilde m^{k,r,q^\star}(x^{k,r,q^\star}),
\end{equation}
and Step~1 selects
\begin{equation}
\label{eq:clamp}
\Delta^{k,r,q^\star}=\min\!\Big\{\max\!\big(\tilde\Delta^{k,r,q^\star},\ \tau\|\tilde g^{k,r,q^\star}\|\big),\ \Delta_{\mathrm{icb}}^{k,r,q^\star}\Big\}.
\end{equation}
From \eqref{eq:FL} with $m=\tilde m^{k,r,q^\star}$ and $\Delta=\tilde\Delta^{k,r,q^\star}$,
\begin{equation}
\label{eq:FLcenterCase1}
\|\nabla h(x^{k,r+1})\|\ \le\ \|\tilde g^{k,r,q^\star}\|+\kappa_{eg}\,\tilde\Delta^{k,r,q^\star}.
\end{equation}

\noindent If the inner maximum in \eqref{eq:clamp} is selected, then
\begin{equation}
\label{eq:innermax_bounds}
\|\tilde g^{k,r,q^\star}\|\ \le\ \frac{\Delta^{k,r+1}}{\tau},
\qquad
\tilde\Delta^{k,r,q^\star}\ \le\ \Delta^{k,r+1}.
\end{equation}
In addition, using \eqref{eq:stop} together with the first inequality in
\eqref{eq:innermax_bounds}, we also obtain
\[
\tilde\Delta^{k,r,q^\star}\ \le\ \mu\,\frac{\Delta^{k,r+1}}{\tau}.
\]
Substituting these bounds into \eqref{eq:FLcenterCase1} gives
\begin{equation}
\label{eq:case1a}
\|\nabla h(x^{k,r+1})\|\ \le\ \frac{1+\kappa_{eg}\mu}{\tau}\,\Delta^{k,r+1}.
\end{equation}

If instead $\Delta^{k,r+1}=\Delta_{\mathrm{icb}}^{k,r,q^\star}$, then at the immediately preceding improvement sub-iterate (index $u-1$) the stopping test is false; with $\alpha\in(0,1)$,
\begin{equation}
\label{eq:prestop}
\mu\,\big\|g^{k,r,q^\star,u-1}\big\|\ <\ \alpha^{\,u-2}\,\Delta^{k,r+1}\ \le\ \Delta^{k,r+1}.
\end{equation}
Applying \eqref{eq:FL} at $\Delta=\Delta^{k,r+1}$ with $m$ equal to the model at sub-iterate $u-1$ gives
\begin{equation}
\label{eq:case1b}
\|\nabla h(x^{k,r+1})\|\ \le\ \big\|g^{k,r,q^\star,u-1}\big\|+\kappa_{eg}\,\Delta^{k,r+1}
\ \le\ \Big(\frac{1}{\mu}+\kappa_{eg}\Big)\,\Delta^{k,r+1}.
\end{equation}

\textbf{Case 2 (Step~\ref{step:1.2b}).}
Here, the incumbent model $m_{\mathrm{icb}}^{k,r,q^\star}$ is already certifiably fully linear on $B(x^{k,r,q^\star};\Delta_{\mathrm{icb}}^{k,r,q^\star})$ and satisfies the criticality relation
\[
\Delta_{\mathrm{icb}}^{k,r,q^\star}\ \le\ \mu\,\|g_{\mathrm{icb}}^{k,r,q^\star}\|.
\]
Hence, without invoking the clamp rule, the fully linear inequality at the center gives
\[
\|\nabla h(x^{k,r+1})\|\ \le\ \|g_{\mathrm{icb}}^{k,r,q^\star}\| + \kappa_{eg}\,\Delta_{\mathrm{icb}}^{k,r,q^\star}.
\]
Using $\Delta_{\mathrm{icb}}^{k,r,q^\star}\le\mu\|g_{\mathrm{icb}}^{k,r,q^\star}\|$, we obtain
\[
\|\nabla h(x^{k,r+1})\|\ \le\ \Big(\kappa_{eg}+\tfrac{1}{\mu}\Big)\,\Delta_{\mathrm{icb}}^{k,r,q^\star}.
\]
Since $\Delta^{k,r+1}=\Delta_{\mathrm{icb}}^{k,r,q^\star}$ at termination, this reduces to
\begin{equation}
\label{eq:case2}
\|\nabla h(x^{k,r+1})\|\ \le\ \Big(\kappa_{eg}+\tfrac{1}{\mu}\Big)\,\Delta^{k,r+1}.
\end{equation}

\noindent Combining \eqref{eq:case1a}, \eqref{eq:case1b}, and \eqref{eq:case2} establishes \eqref{eq:finalC}.
\end{proof}

\begin{remark}

If $\varepsilon^{k,r}_4$ is chosen too large, termination could occur in Step~\ref{step:1.1}, where the model need not be fully linear, and the analysis would fail. 
By contrast, taking $\varepsilon^{k,r}_4$ sufficiently small ensures that any radius satisfying $\Delta^{k,r,q}\leq\varepsilon^{k,r}_4$ is itself small, which in turn forces the associated model gradient $\|g_{\mathrm{icb}}^{k,r,q}\|$ to also be sufficiently small, guaranteeing that termination occurs in the criticality branch (Step~\ref{step:1.2}). 
\end{remark}

\begin{lemma}[Convergence of inner iterations for the smooth case]\label{lem:inner-exact}
Suppose that Assumptions~\ref{ass:smooth-lip-gradient} and~\ref{ass:hessian-bounded} as well as Assumptions~\ref{ass:fi-and-g-lower-bounded} and~\ref{ass:closed_form_solution} hold, with $\rho^k = 2\beta^k$. If the $x_i$ subproblems in Algorithm  \ref{proposed_alg} are solved inexactly, satisfying the following inner iteration criteria:

\begin{equation}
\begin{aligned}
\epsilon_1^k \geq \epsilon_1^{k,r} &:= \left\| \rho^k A^\top \big( B\overline{x}^{k,r+1} + z^{k,r+1} - B\overline{x}^{k,r} - z^{k,r} \big) \right\|, \\
\epsilon_2^k \geq \epsilon_2^{k,r} &:= \left\| \rho^k B^\top \big( z^{k,r+1} - z^{k,r} \big) \right\|, \\
\epsilon_3^k \geq \epsilon_3^{k,r} &:= \left\| A x^{k,r+1} + B\overline{x}^{k,r+1} + z^{k,r+1} \right\|, \\
\epsilon_4^k \geq \epsilon_4^{k,r} &\geq \Delta^{k,r+1};
\end{aligned}
\label{eq:inner_stop_criteria}
\end{equation}

 \noindent then the inner iterations at a finite $r$ satisfy the following criteria:

\begin{equation}
\begin{aligned}
d_1^k + d_4^k  =&  \nabla f(x^{k,r+1}) + A^\top y^{k,r+1}, \\
d_2^k        =& \nabla g(\bar{x}^{k,r+1}) + B^\top y^{k,r+1}, \\
0            & =     \lambda^k + \beta^k z^{k,r+1} + y^{k,r+1}, \\
d_3^k        & =     A x^{k,r+1} + B \bar{x}^{k,r+1} + z^{k,r+1},
\end{aligned}
\label{eq:inner_optimality}
\end{equation}

\noindent for some \(d_1^k, d_2^k, d_3^k,\) and \(d_4^k\) satisfying 
\(\|d_1^k\| \leq \epsilon_1^k\), \(\|d_2^k\| \leq \epsilon_2^k\), 
\(\|d_3^k\| \leq \epsilon_3^k\), and 
\(\|d_4^k\| \leq C_1(\tau,\mu,\kappa_{eg})\,\epsilon_4^k\), 
respectively, where \(C_1(\tau,\mu,\kappa_{eg}) \geq 0\) is the constant 
defined in Lemma~\ref{lem:grad_norm_tr_bound}.

\end{lemma}

\begin{proof}

\noindent Following the analysis in \cite{tang2022}, we first prove that:
\begin{equation}
\begin{split}
L(x^{k+1}, \overline{x}^{k,r+1}, z^{k,r+1}, y^{k,r+1}) &\leq L(x^{k,r}, \overline{x}^{k,r}, z^{k,r}, y^{k,r}) \\
&\quad - \beta^k \left\| B\overline{x}^{k,r+1} - B\overline{x}^{k,r} \right\|^2 - \frac{\beta^k}{2} \left\| z^{k,r+1} - z^{k,r} \right\|^2 
\end{split}
\label{eq:lagrangian_decrease}
\end{equation}

\noindent for $r=0,1,2,\ldots$. Note that for the $x_i$ subproblem in Algorithm~\ref{proposed_alg}, the following holds as the update of $x_i$ does not lead to an increase in $L$:

\begin{equation}
\begin{aligned}
L(x_1^{k,r}, x_2^{k,r}, \ldots, x_N^{k,r}, \overline{x}^{k,r}, z^{k,r}, y^{k,r}) &\geq L(x_1^{k,r+1}, x_2^{k,r}, \ldots, x_N^{k,r}, \overline{x}^{k,r}, z^{k,r}, y^{k,r}) \\
&\geq L(x_1^{k,r+1}, x_2^{k,r+1}, \ldots, x_N^{k,r}, \overline{x}^{k,r}, z^{k,r}, y^{k,r}) \\
&\geq L(x_1^{k,r+1}, x_2^{k,r+1}, \ldots, x_N^{k,r+1}, \overline{x}^{k,r}, z^{k,r}, y^{k,r}). 
\end{aligned}
\label{eq:sequential_decrease}
\end{equation}

\noindent This monotonicity holds even though the $x_i$-subproblems are solved approximately because 
the trust-region acceptance test ensures that each accepted step produces nonnegative actual reduction in $L$, 
while rejected steps leave the iterate unchanged. 
Hence every block update guarantees $L$ does not increase.

For simplicity, we concatenate the variables $x_i$ ($i=1,\ldots,N$) into a single vector $x$, and obtain from \eqref{eq:sequential_decrease}:

\begin{equation}\label{eq:concatenated_decrease}
L(x^{k,r}, \overline{x}^{k,r}, z^{k,r}, y^{k,r}) \geq L(x^{k,r+1}, \overline{x}^{k,r}, z^{k,r}, y^{k,r}).
\end{equation}

\noindent Next, we consider the decrease in the $\overline{x}$ update. Using the optimality condition of the $\overline{x}$ subproblem and the convexity of \(g\), \cite{tang2022} showed that:
\begin{equation}
\begin{aligned}
L(x^{k,r+1}, \overline{x}^{k,r+1}, z^{k,r}, y^{k,r}) &\leq L(x^{k,r+1}, \overline{x}^{k,r}, z^{k,r}, y^{k,r}) \\
&\quad - \frac{\rho^k}{2} \left\| B\overline{x}^{k,r+1} - B\overline{x}^{k,r} \right\|^2.
\end{aligned}
\label{eq:xbar_update_decrease}
\end{equation}

\noindent Now consider the $z$- and $y$-updates. \cite{tang2022} obtained the following bound:
\begin{equation}
\begin{aligned}
L(x^{k,r+1}, \overline{x}^{k,r+1}, z^{k,r+1}, y^{k,r+1}) &- L(x^{k,r+1}, \overline{x}^{k,r+1}, z^{k,r}, y^{k,r}) \\
&\leq -\left(\frac{\rho^k}{2} - \frac{(\beta^k)^2}{\rho^k}\right) \|z^{k,r+1} - z^{k,r}\|^2.
\end{aligned}
\label{eq:z_y_update_bound_0}
\end{equation}

\noindent Setting \(\rho^k = 2\beta^k\), we have:
\begin{equation}
\begin{aligned}
L(x^{k,r+1}, \overline{x}^{k,r+1}, z^{k,r+1}, y^{k,r+1}) &- L(x^{k,r+1}, \overline{x}^{k,r+1}, z^{k,r}, y^{k,r}) \\
&= -\frac{\beta^k}{2} \|z^{k,r+1} - z^{k,r}\|^2.
\end{aligned}
\label{eq:z_y_update_bound}
\end{equation}

\noindent By combining inequalities \eqref{eq:concatenated_decrease}, \eqref{eq:xbar_update_decrease}, and \eqref{eq:z_y_update_bound}, we arrive at the result in \eqref{eq:lagrangian_decrease}.

We proceed by establishing that the augmented Lagrangian is bounded below, thereby guaranteeing its convergence to some scalar \(L^k \in \mathbb{R}\). We first define \(v(z; \lambda, \beta) = \lambda^\top z + \frac{\beta}{2} \|z\|^2\). Since \(v\) is strongly convex with modulus \(\beta\) and \(\rho^k >\beta^k\), it follows that for any \(z' \in \mathbb{R}^n\):
\begin{equation}
v(z^{k,r}; \lambda^k, \beta^k) + (\lambda^k + \beta^k z^{k,r})^\top (z' - z^{k,r}) + \frac{\rho^k}{2} \|z' - z^{k,r}\|^2 \geq v(z'; \lambda^k, \beta^k). \label{eq:v_convexity}
\end{equation}

\noindent From the \(z\)-update in Algorithm~\ref{proposed_alg}, we have:
\begin{equation}
\lambda^k + \beta^k z^{k,r+1} = -y^{k,r+1}. \label{eq:from_z_update}
\end{equation}

\noindent Combining \eqref{eq:v_convexity} and \eqref{eq:from_z_update}, we obtain:
\begin{equation}
v(z^{k,r}; \lambda^k, \beta^k) + y^{k,r\,\top} (z^{k,r} - z') \geq v(z'; \lambda^k, \beta^k) - \frac{\rho^k}{2} \|z' - z^{k,r}\|^2. \label{eq:v_inequality}
\end{equation}

\noindent Letting \(z' = -(A x^{k,r} + B \overline{x}^{k,r})\), we have:
\begin{equation}
\begin{aligned}
v(z^{k,r}; \lambda^k, \beta^k) 
&+ y^{k,r\,\top}\!\left(A x^{k,r} + B \overline{x}^{k,r} + z^{k,r}\right) \\
&\geq\;
v\!\left(-\!\left(A x^{k,r} + B \overline{x}^{k,r}\right); \lambda^k, \beta^k\right)
 - \tfrac{\rho^k}{2}\,\|A x^{k,r} + B \overline{x}^{k,r} + z^{k,r}\|^2.
\end{aligned}
\label{eq:lagrangian_lower_bound}
\end{equation}

\noindent Consequently, the augmented Lagrangian satisfies:
\begin{equation}
\begin{aligned}
L(x^{k,r}, \overline{x}^{k,r}, z^{k,r}, y^{k,r}) &= f(x^{k,r}) + g(\overline{x}^{k,r}) + v(z^{k,r}; \lambda^k, \beta^k) \\
&\quad + y^{k,r\,\top} (A x^{k,r} + B \overline{x}^{k,r} + z^{k,r}) + \frac{\rho^k}{2} \left\| A x^{k,r} + B \overline{x}^{k,r} + z^{k,r} \right\|^2 \\
&\geq f(x^{k,r}) + g(\overline{x}^{k,r}) + v(-(A x^{k,r} + B \overline{x}^{k,r}); \lambda^k, \beta^k).
\end{aligned}
\label{eq:auglag_expression}
\end{equation}

\noindent The inequality \(v(z) = \lambda^\top z + \frac{\beta}{2} \|z\|^2 \geq -\frac{\|\lambda\|^2}{2\beta}\), together with bounded \(\lambda \in [\underline{\lambda}, \overline{\lambda}]\), \(\beta^k \geq \beta^1 > 0\), and lower-bounded \(f\) and \(g\), implies that the augmented Lagrangian \(L\) has a lower bound.

Taking the limit as \(r \to \infty\) in \eqref{eq:lagrangian_decrease}, 
we find \(B\overline{x}^{k,r+1} - B\overline{x}^{k,r} \to 0\), 
\(z^{k,r+1} - z^{k,r} \to 0\), and 
\(A x^{k,r} + B \overline{x}^{k,r} + z^{k,r} \to 0\). 
Moreover, by Lemma~\ref{dfo_smooth_alg}, for each \(r\) the trust-region radius 
can be reduced arbitrarily, so there exists a finite \(q^\star\) such that 
\(\Delta^{k,r+1}:=\Delta^{k,r,q^\star} \leq \epsilon_4^{k,r}\). 
Thus, there exists some finite \(r\) satisfying \eqref{eq:inner_stop_criteria}, 
at which point the optimality condition for \(x^{k,r+1}\) is:
\begin{equation}
d_4^k = \nabla f(x^{k,r+1}) + A^\top y^{k,r} 
+ \rho^k A^\top \big(A x^{k,r+1} + B \overline{x}^{k,r} + z^{k,r}\big).
\label{eq:optimality_condition}
\end{equation}

\noindent Here, \(d_4^k\) represents the gradient error from inexactly solving the \(x_i\) subproblems using trust-region solvers in Algorithm~\ref{proposed_alg} .

Using the \(y^{k,r}\) update, \eqref{eq:optimality_condition} can be rewritten as:
\begin{equation}
d_4^k =\nabla f(x^{k,r+1}) + A^\top y^{k,r+1} - \rho^k A^\top (B \overline{x}^{k,r+1} + z^{k,r+1} - B \overline{x}^{k,r} - z^{k,r}), \label{eq:rewritten_optimality}
\end{equation}
which implies:
\begin{equation}
d_4^k + \underbrace{\rho^k A^\top (B \overline{x}^{k,r+1} + z^{k,r+1} - B \overline{x}^{k,r} - z^{k,r})}_{d_1^k} =\nabla f(x^{k,r+1}) + A^\top y^{k,r+1}. \label{eq:subgradient_inclusion}
\end{equation}

\noindent From the first inequality in \eqref{eq:inner_stop_criteria}, the norm of \(d_1^k\) in \eqref{eq:subgradient_inclusion} is bounded by \(\epsilon_1^k\). For \(d_4^k\), Lemma~\ref{lem:grad_norm_tr_bound} ensures \(\|d_4^k\| \leq C_1(\tau,\mu,\kappa_{eg})\,\Delta^{k,r+1}\) with \(C_1(\tau,\mu,\kappa_{eg})>0\). Since \(\Delta^{k,r+1}\leq \epsilon_4^{k,r}\leq \epsilon_4^k\), it follows that \(\|d_4^k\| \leq C_1(\tau,\mu,\kappa_{eg})\,\epsilon_4^k\), thereby establishing the first condition in \eqref{eq:inner_optimality}.

To derive the second condition in \eqref{eq:inner_optimality}, consider the optimality condition for the \(\overline{x}\) update:
\begin{equation}
0 = \nabla g(\overline{x}^{k,r+1}) + \rho^k B^\top \left( A x^{k,r+1} + B \overline{x}^{k,r+1} + z^{k,r} + \frac{y^{k,r}}{\rho^k} \right). \label{eq:optimality_condition_for_xbar}
\end{equation}

\noindent Combining \eqref{eq:optimality_condition_for_xbar} with the update rule for \(y^{k,r}\), we have:
\begin{equation}
\underbrace{\rho^k B^\top (z^{k,r+1} - z^{k,r})}_{d_2^k} = \nabla g(\overline{x}^{k,r+1}) + B^\top y^{k,r+1}. \label{eq:subgradient_inclusion_2}
\end{equation}

\noindent From the second inequality in \eqref{eq:inner_stop_criteria}, the norm of \(d_2^k\) is bounded by \(\epsilon_2^k\), which establishes the second condition in \eqref{eq:inner_optimality}. The third condition arises from \eqref{eq:from_z_update}, and the fourth is obvious.

\end{proof}

\begin{remark}
The use of Lemma~\ref{lem:grad_norm_tr_bound} in the above proof relies on 
$\varepsilon^{k,r}_4$ being sufficiently small, since only then does termination occur in 
Step~\ref{step:1.2} with a certifiably fully linear model. 
At early outer iterations, when $\varepsilon^{k,r}_4$ may still be large, the bound 
$\|d_4^k\|\le C_1(\tau,\mu,\kappa_{eg})\,\Delta^{k,r+1}$ cannot be theoretically guaranteed. 
Nevertheless, the inner termination condition $\Delta^{k,r+1}\leq \varepsilon^{k,r}_4$ 
provides a practical stopping criterion and still ensures progress. 
As the algorithm proceeds and $\varepsilon^{k,r}_4$ decreases, the assumptions of 
Lemma~\ref{lem:grad_norm_tr_bound} are eventually satisfied, and the gradient--radius bound holds. 
This does not compromise the overall convergence, since the convergence guarantees are driven 
by later iterations where the tolerances are sufficiently small.
\end{remark}

For the convergence of the outer iterations, it is desirable to assume an upper bound on the augmented Lagrangian as stated below.

\begin{assumption}\label{ass:aug-lag-bounded}
There exists a constant upper bound $\bar{L}$ such that the value of the augmented Lagrangian at the initial points of each inner iteration, $L(x^{k,0},\, \bar{x}^{k,0},\, z^{k,0},\, y^{k,0},\, \lambda^k,\, \rho^k,\, \beta^k)$, does not exceed $\bar{L}$ for any outer iteration $k$. 
\end{assumption}

This assumption serves as a type of “warm start,” ensuring that the augmented Lagrangian remains finite at initialization. This can be achieved by initializing with a feasible point, which keeps the starting value of the augmented Lagrangian finite and controlled at each iteration.

\begin{lemma}[Convergence of outer iterations for the smooth case]\label{lem:outer-exact} 
 Suppose that Assumptions~\ref{ass:fi-and-g-lower-bounded}, ~\ref{ass:closed_form_solution}, and~\ref{ass:aug-lag-bounded} hold. Then for any \(\epsilon_1\), \(\epsilon_2\),\(\epsilon_3\), and \(\epsilon_4 > 0\), within a finite number of outer iterations \(k\), Algorithm 1 finds an approximate stationary point \((x^{k+1}, \overline{x}^{k+1}, z^{k+1}, y^{k+1})\) of \eqref{eq:standard_admm_form}, satisfying
\begin{equation}
\begin{aligned}
d_1+d_4 &= \nabla f(x^{k+1}) +  A^T y^{k+1}, \\
d_2 &= \nabla g(\overline{x}^{k+1})  + B^T y^{k+1}, \\
d_3 &= A x^{k+1} + B \overline{x}^{k+1}
\end{aligned}
\label{eq:outer_optimality}
\end{equation}

\noindent for some \(d_1\), \(d_2\), \(d_3\), \(d_4\), satisfying \(\|d_j\| \leq \epsilon_j\), \(j = 1, 2, 3\) and \(\|d_4\| \leq C_1\epsilon_4\) for some $C_1>0$.

\end{lemma}

\begin{proof}
Following the argument in~\cite{tang2022}, consider the case where \( \beta^k \) is unbounded. From \eqref{eq:auglag_expression}, we obtain
\begin{equation}
\overline{L} \geq f(x^{k+1}) + g(\overline{x}^{k+1}) - \lambda^k (A x^{k+1} + B x^{k+1}) + \frac{\beta^k}{2} \|A x^{k+1} + B x^{k+1}\|^2. \label{eq:auglag_bound}
\end{equation}

\vspace{-0.2cm}

\noindent Since \( f \) and \( g \) are lower bounded, the right-hand side remains bounded only if \( A x^{k+1} + B x^{k+1} \to 0 \) as \( \beta^k \to \infty \). In combination with the first two conditions in \eqref{eq:inner_optimality}, satisfied in the limit as \( \epsilon_1^k, \epsilon_2^k, \epsilon_3^k, \epsilon_4^k  \to 0 \), this implies \eqref{eq:outer_optimality}.

Now consider the case where \( \beta^k \) is bounded, meaning the update \( \beta^{k+1} = \gamma \beta^k \) occurs only finitely many times. According to the outer iteration rule in Algorithm~1, we then have \( \|z^{k+1}\| \leq \omega \|z^k\| \) for all but finitely many \( k \), implying \( z^{k+1} \to 0 \). As a result, taking the limit in \eqref{eq:inner_optimality} leads to \eqref{eq:outer_optimality}.
\end{proof}

\vspace{-0.1cm}

We note that for the nonsmooth case, a direct  per-iterate bound between the trust-region radius and a Clarke-stationarity measure is not available in closed form. Unlike the smooth case where fully linear models yield a gradient–radius bound (Lemma~\ref{lem:grad_norm_tr_bound}), our nonsmooth TR models are not known to satisfy a Clarke “fully linear’’ property. To enable analysis, we therefore adopt the following ad hoc assumption, paralleling Lemma~\ref{lem:grad_norm_tr_bound}.

\begin{assumption}[Clarke Subgradient Radius Bound]
\label{assump:subgrad-radius}
For each inexact nonsmooth trust-region subproblem solved with Algorithm~\ref{dfo_nonsmooth_alg}, 
there exists a constant $C_2 > 0$ such that whenever the solver terminates at radius $\Delta^{k,r+1}$, we obtain
\[
\bar{d}_4^k \coloneqq \operatorname{dist}\!\left(0,\,\partial_C h(x^{k,r+1})\right)
\;\le\; C_2\,\Delta^{k,r+1}.
\]
\end{assumption}

Assumption~\ref{assump:subgrad-radius} states that the trust-region radius bounds the Clarke subgradient norm at termination, analogous to the smooth gradient–radius relation. While a proof of this bound is not yet available for the nonsmooth trust-region method, it is consistent with its asymptotic properties $\lim_{q\to\infty}\Delta_q=0$ and $\liminf_{q\to\infty}\operatorname{dist}(0,\partial_C h(x_q))=0$. This assumption allows us to extend the ADMM–DFO convergence results to nonsmooth problems.

\begin{lemma}[Convergence of inner iterations for the nonsmooth case]\label{lem:inner-exact_nonsmooth}

Suppose Assumptions~\ref{ass:local-lipschitz} and \ref{ass:surrogate-growth} 
(for the nonsmooth trust-region solver) as well as 
Assumptions~\textup{\ref{ass:fi-and-g-lower-bounded}--\ref{assump:subgrad-radius}} 
hold, with $\rho^k = 2\beta^k$. If the $x_i$ subproblems in 
Algorithm~\ref{proposed_alg} are solved inexactly using 
Algorithm~\ref{dfo_nonsmooth_alg} and are terminated at an iterate $q^\star$ 
that satisfies \eqref{eq:inner_stop_criteria}, then the inner iterations 
satisfy, for some finite $r$, the following conditions:

\begin{equation}
\begin{aligned}
\bar d_1^k + \bar d_4^k  &\in \partial_C f(x^{k,r+1}) + A^\top y^{k,r+1}, \\
\bar d_2^k               &\in \partial g(\bar{x}^{k,r+1}) + B^\top y^{k,r+1}, \\
0                        &= \lambda^k + \beta^k z^{k,r+1} + y^{k,r+1}, \\
\bar d_3^k               &= A x^{k,r+1} + B \bar{x}^{k,r+1} + z^{k,r+1},
\end{aligned}
\label{eq:inner_optimality_nonsmooth}
\end{equation}

\noindent
for some $\bar{d}_1^k$, $\bar{d}_2^k$, $\bar{d}_3^k$, and $\bar{d}_4^k$ satisfying
$\|\bar{d}_j^k\| \le \epsilon_j^k$ $(j=1,2,3)$ and
$\|\bar{d}_4^k\| \le C_2\,\epsilon_4^k$ with $C_2 > 0$ from
Assumption~\ref{assump:subgrad-radius}.

\end{lemma}

The proof of Lemma~\ref{lem:inner-exact_nonsmooth} has been omitted for brevity as it follows the same reasoning 
as that of Lemma~\ref{lem:inner-exact}, replacing gradient terms by 
Clarke subgradients and invoking 
Assumption~\ref{assump:subgrad-radius} instead of 
Lemma~\ref{lem:grad_norm_tr_bound}.

\begin{lemma}[Convergence of outer iterations for the nonsmooth case]\label{lem:outer-exact_nonsmooth} 
 Suppose that Assumptions~\ref{ass:fi-and-g-lower-bounded}, ~\ref{ass:closed_form_solution}, and~\ref{ass:aug-lag-bounded} hold. Then for any \(\epsilon_1\), \(\epsilon_2\),\(\epsilon_3\), and \(\epsilon_4 > 0\), within a finite number of outer iterations \(k\), Algorithm 1 finds an approximate Clarke stationary point \((x^{k+1}, \overline{x}^{k+1}, z^{k+1}, y^{k+1})\) of \eqref{eq:standard_admm_form}, satisfying
\begin{equation}
\begin{aligned}
\bar d_1+\bar d_4 &\in \partial_C f(x^{k+1}) +  A^\top y^{k+1}, \\
\bar d_2 &\in \partial g(\overline{x}^{k+1})  + B^\top y^{k+1}, \\
\bar d_3 &= A x^{k+1} + B \overline{x}^{k+1}
\end{aligned}
\label{eq:outer_optimality_nonsmooth}
\end{equation}

\noindent for some $\bar d_1, \bar d_2, \bar d_3, \bar d_4$ satisfying 
$\|\bar d_i\| \leq \epsilon_i$, $i=1,2,3$, $\|\bar{d}_4^k\| \le C_2\,\epsilon_4^k$ with $C_2 > 0$ .

\end{lemma}

\noindent The proof of the above lemma follows the same logic as the proof of Lemma \ref{lem:outer-exact} and is omitted for brevity.

The conclusions of the convergence lemmas for both the smooth and nonsmooth 
cases are summarized in the following theorem.

\begin{theorem}[Global convergence of inexact DFO-ADMM algorithm]
Suppose Assumptions~\textup{\ref{ass:fi-and-g-lower-bounded}--\ref{ass:aug-lag-bounded}} hold. 
Then, for any given tolerances $\epsilon_1,\,\epsilon_2,\,\epsilon_3,\,\epsilon_4 > 0$, 
Algorithm~\ref{proposed_alg} produces, within finitely many inner and outer iterations, 
an approximate stationary point $(x_\epsilon,\,\bar{x}_\epsilon,\,z_\epsilon,\,y_\epsilon)$, 
with the following cases:
\begin{itemize}
    \item[(i)] If the $x_i$-subproblems are solved using the smooth trust-region method 
    (Algorithm~\ref{dfo_smooth_alg}) under 
    Assumptions~\ref{ass:smooth-lip-gradient}--\ref{ass:hessian-bounded}, 
    the limit point is an approximate KKT point satisfying 
    the conditions in~\eqref{eq:outer_optimality}.
    \item[(ii)] If the $x_i$-subproblems are solved using the nonsmooth trust-region method 
    (Algorithm~\ref{dfo_nonsmooth_alg}) under 
    Assumptions~\ref{ass:local-lipschitz}--\ref{ass:surrogate-growth}, and further \ref{assump:subgrad-radius} holds for the nonsmooth algorithm, the limit point is an approximate Clarke stationary point satisfying 
    the conditions in~\eqref{eq:outer_optimality_nonsmooth}.
\end{itemize}
\end{theorem}

\subsection{Parameter Tuning}
We recommend the following parameter setting strategy for Algorithm~\ref{proposed_alg}. 
The inner and outer tolerances are updated adaptively as
\begin{equation}\label{eq:parameter_tuning}
\begin{aligned}
&\epsilon_1^k=\max\!\left(\tfrac{c_1}{a_1^{\,k}},\,\epsilon_1\right), \qquad &&\epsilon_4^k=\epsilon_1^k,\\[0.25em]
&\epsilon_2^k=\max\!\left(\tfrac{c_2}{a_2^{\,k}},\,\epsilon_2\right), \qquad &&\epsilon_3^k=\tfrac{\epsilon_3}{\epsilon_1}\,\epsilon_1^k,\\[0.25em]
&\epsilon_4^{k,r}=\max\!\left(c_4\,(\epsilon_1^{k,r})^{a_4},\,\epsilon_4^k\right),
\end{aligned}
\end{equation}
with \(a_1,a_2,a_4\in[1,2]\) and \(c_1,c_2,c_4\in[1,2]\). 
The specific values of \(a_i\) and \(c_i\) influence only the relative inner--outer iteration counts, 
while overall convergence is preserved.

The outer parameters \(\omega\) (slack decay threshold) and \(\gamma\) (penalty growth) 
control the trade-off between speed and conditioning: 
smaller \(\omega\) or larger \(\gamma\) accelerate slack reduction but can worsen conditioning 
and increase inner cost, whereas conservative values slow outer progress. 
In practice, intermediate choices balance this trade-off effectively. 
Parameter settings for the underlying trust-region DFO solvers may follow the recommendations 
in \cite{conn2009global} (smooth) and \cite{Liuzzi2019} (nonsmooth).

\section{Numerical Examples}\label{sec:numerical}
This section presents numerical case studies for our two-level inexact ADMM–DFO framework, compared against standard monolithic derivative-free baselines under identical stopping criteria and increasing dimensions. We report implementation details, wall-clock time, function evaluations, and accuracy, and include a distributed training problem in nonconvex, nonsmooth settings. Reproducible code is available at our GitHub repository.\footnote{\url{https://github.com/dfasiku/Two-level-inexact-admm-dfo}}

\subsection{A convex and smooth problem}

We first consider the minimization of the $N$-dimensional ARWHEAD function, a benchmark test problem in derivative-free optimization taken from \cite{Conn1994}. The function is defined as
\begin{equation}
f(u) = \sum_{i=1}^{N-1} \left[\, (u_i^2 + u_N^2)^2 - 4u_i + 3 \,\right],
\label{eq:arrowhead_function}
\end{equation}
which is smooth and convex, with a global minimum $f^* = 0$ attained at $u_i = 1$ for $i = 1,\ldots,N-1$ and $u_N = 0$. The problem is inherently non-separable because the last variable $u_N$ appears in every quartic term $(u_i^2 + u_N^2)^2$, thereby coupling $u_N$ with all other variables simultaneously. We assume that only the zero-order information of the function is available necessitating the use of a derivative-free solver.

To demonstrate the scalability limitations of monolithic approaches, we benchmark two established DFO solvers. The first, \textsc{NEWUOA} \cite{Powell2006}, accessed through the NLopt library, is a trust-region-based method known for excellent practical performance on smooth problems in moderate to large dimensions, though it lacks theoretical convergence guarantees. As a baseline, we also test the Nelder--Mead method via the \texttt{scipy.optimize.minimize} interface in Python. All runs are initialized from the zero vector and terminated either upon reaching an objective accuracy of $10^{-5}$ or exhausting a budget of $10^6$ function evaluations. Computations are carried out on a 12-core (24-thread) Intel processor @ 2.11\,GHz with 64~GB RAM, running a 64-bit OS and Python~3.13. Table~\ref{tab:arwhead_results} reports the performance across increasing problem dimensions. While NEWUOA consistently reached the target accuracy, its runtime grew rapidly with dimension. Nelder--Mead converged only up to $N=100$, after which it exhausted the evaluation budget with poor solutions, illustrating the curse of dimensionality in classical DFO methods.

We then apply our proposed ADMM--DFO framework (Algorithm~\ref{proposed_alg}), which decomposes the $N$-dimensional Arrowhead problem into $N_B$ blocks where each block contains four block-local variables and a local copy of a globally shared variable (with the last block possibly smaller). This block structure renders the global objective separable while introducing linear consensus constraints consistent with the formulation in \eqref{eq:sep_opt_problem}. The constraint matrices $A$ and $B$ are chosen so that the last coordinate of each block equals the shared global variable, thereby enforcing consensus across all blocks. The algorithm was parameterized with initial penalty $\beta^1 = 20$, adaptive tolerance parameters $c_1=c_2=c_4=1$, $a_1=a_2=2$, $a_4=1.5$, final tolerances $\epsilon_1=\epsilon_2=\epsilon_3=\epsilon_4=10^{-5}$, and penalty update factors $\omega = 0.75$, $\gamma = 1.005$. Primal and dual variables are initialized uniformly in $[0,1]$. The block subproblems are solved in parallel across the 24 available threads using the \texttt{concurrent.futures.ThreadPoolExecutor} in Python~3.13t; for $N_B>24$, the subproblems are processed in batches of 24. Because function evaluations occur blockwise in parallel on a partially separable objective with overlaps, we report for ADMM--DFO the \emph{maximum per-block evaluations} as a proxy for the monolithic function evaluation count, which conservatively upper-bounds the serial total while reflecting parallel execution. 

As seen in Table~\ref{tab:arwhead_results}, at small dimensions ADMM--DFO underperforms compared to monolithic solvers because distributed overhead dominates. However, as the dimension grows, ADMM--DFO scales effectively, yielding substantial speedups in run time with controlled function evaluations. For $N=1200$, Figure~\ref{fig:arwhead_convergence} shows the convergence of the augmented Lagrangian, $\epsilon_1^{k,r}$, $\epsilon_2^{k,r}$, and $\epsilon_3^{k,r}$ across outer iterations, confirming convergence to a stationary point. This demonstrates our framework's ability to handle non-separable problems with complex variable interactions while providing convergence guarantees in high-dimensional and derivative-free settings.

% ===== Table: ARWHEAD results 
\FloatBarrier
\begin{table*}[t]
\centering
\caption{Results for the ARWHEAD function with Nelder--Mead, NEWUOA, and ADMM--DFO solvers. Bold indicates the least computational time and least number of function evaluations (\#Fev) per row. $\dagger$: The maximum number of evaluations was reached.}
\label{tab:arwhead_results}

\begingroup
\setlength{\tabcolsep}{2.5pt}
\renewcommand{\arraystretch}{1.0}
\small

\begin{adjustbox}{max width=\textwidth}
\begin{tabular}{cccccccccccc}
\toprule
\multirow{2}{*}{$N$}
& \multicolumn{4}{c}{Nelder--Mead}
& \multicolumn{4}{c}{NEWUOA}
& \multicolumn{3}{c}{ADMM--DFO} \\
\cmidrule(lr){2-5} \cmidrule(lr){6-9} \cmidrule(lr){10-12}
& Final Obj. & Time (s) & \#Fev & ~
& Final Obj. & Time (s) & \#Fev & ~
& Final Obj. & Time (s) & \#Fev \\
\midrule
10   & $8.51\times10^{-6}$ & 0.06   & 3{,}433 & ~
     & $8.91\times10^{-6}$ & \textbf{0.00} & \textbf{111} & ~
     & $5.45\times10^{-9}$ & 6.11   & 3{,}548 \\
50   & $9.57\times10^{-6}$ & 2.45   & 105{,}785 & ~
     & $9.43\times10^{-6}$ & \textbf{0.10} & \textbf{974}  & ~
     & $1.56\times10^{-7}$ & 12.61  & 5{,}299 \\
100  & $9.99\times10^{-6}$ & 26.17  & 689{,}259 & ~
     & $9.60\times10^{-6}$ & \textbf{0.37} & \textbf{1{,}487} & ~
     & $5.86\times10^{-6}$ & 18.37  & 4{,}659 \\
200  & $2.05\times10^{2}$  & 69.35  & $\dagger$ & ~
     & $9.95\times10^{-6}$ & \textbf{4.84} & \textbf{3{,}205} & ~
     & $7.30\times10^{-7}$ & 55.83  & 7{,}111 \\
300  & $5.14\times10^{2}$  & 606.33 & $\dagger$ & ~
     & $9.81\times10^{-6}$ & \textbf{19.27} & \textbf{5{,}474} & ~
     & $6.27\times10^{-6}$ & 65.00  & 6{,}865 \\
400  & $1.11\times10^{3}$  & 986.18 & $\dagger$ & ~
     & $9.95\times10^{-6}$ & \textbf{63.57} & 9{,}940 & ~
     & $2.03\times10^{-6}$ & 131.70 & \textbf{9{,}135} \\
500  & $1.48\times10^{3}$  & 1{,}498.29 & $\dagger$ & ~
     & $9.98\times10^{-6}$ & 175.51 & 15{,}483 & ~
     & $3.17\times10^{-7}$ & \textbf{145.00} & \textbf{9{,}399} \\
600  & $1.79\times10^{3}$  & 2{,}449.33 & $\dagger$ & ~
     & $1.00\times10^{-5}$ & 299.29 & 15{,}993 & ~
     & $6.67\times10^{-6}$ & \textbf{204.00} & \textbf{10{,}227} \\
700  & $2.10\times10^{3}$  & 3{,}299.42 & $\dagger$ & ~
     & $9.99\times10^{-6}$ & 661.38 & 24{,}562 & ~
     & $2.90\times10^{-6}$ & \textbf{263.51} & \textbf{11{,}564} \\
1000 & $3.00\times10^{3}$  & 6{,}635.10 & $\dagger$ & ~
     & $9.98\times10^{-6}$ & 3{,}080.57 & 37{,}518 & ~
     & $2.20\times10^{-6}$ & \textbf{471.86} & \textbf{11{,}564} \\
1200 & $3.60\times10^{3}$  & 9{,}471.60 & $\dagger$ & ~
     & $9.98\times10^{-6}$ & 5{,}968.59 & 54{,}906 & ~
     & $2.65\times10^{-6}$ & \textbf{541.15} & \textbf{13{,}429} \\
\bottomrule
\end{tabular}
\end{adjustbox}
\endgroup
\end{table*}

\begin{figure}[htbp]
\centering
\includegraphics[width=0.6\textwidth]{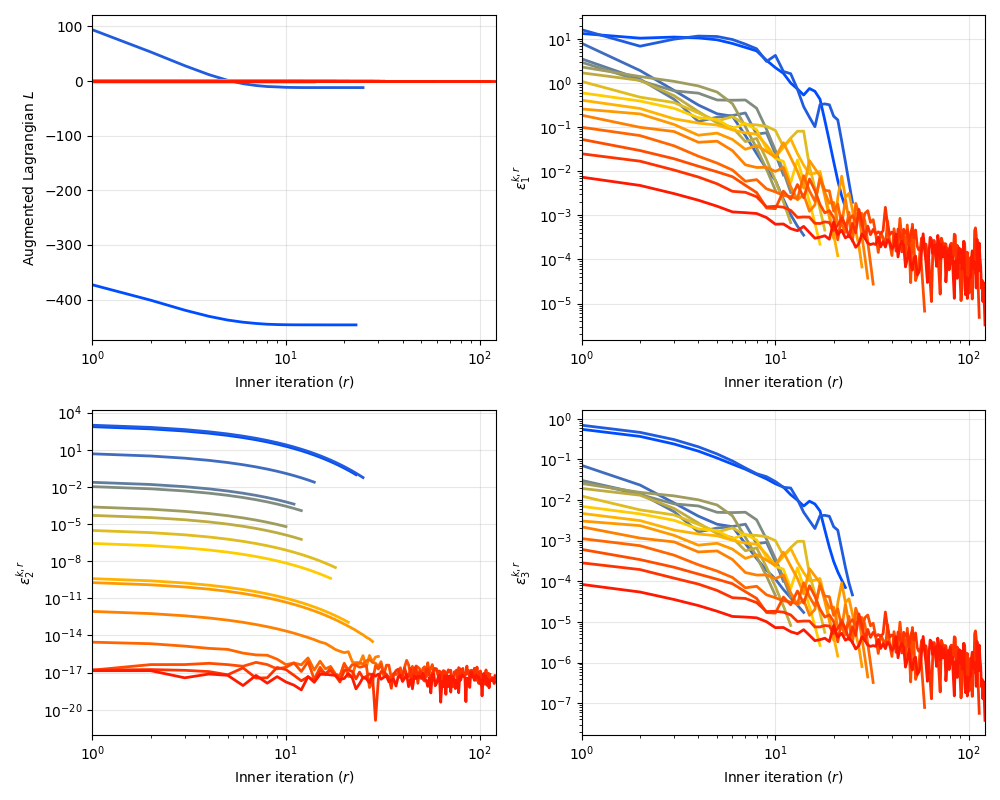}
\caption{Convergence of the augmented Lagrangian and residual norms for the 1200-dimensional ARWHEAD problem. Colors indicate progression of outer iterations from blue (initial) to red (final).}
\label{fig:arwhead_convergence}
\end{figure}

\vspace{5mm}
\subsection{A nonconvex and smooth problem}

We next consider the $N$-dimensional Rosenbrock function, a classical smooth but nonconvex benchmark:
\begin{equation}
f(u) \;=\; \sum_{i=1}^{N-1} \big[\, 100\,(u_{i+1}-u_i^2)^2 + (1-u_i)^2 \,\big],
\label{eq:rosenbrock_function}
\end{equation}
which has global minimizer $u^\star=(1,\ldots,1)$ with $f^\star=0$. The problem exhibits chain-structured coupling through the terms $100\,(u_{i+1}-u_i^2)^2$, making it inherently non-separable. 

Table~\ref{tab:rosen_results} reports the performance of centralized solvers under the same setup as the Arrowhead case (zero initialization, budget of $10^6$ function evaluations, target accuracy $10^{-5}$, same computing environment). NEWUOA consistently reached the required accuracy but with rapidly increasing runtime, while Nelder--Mead failed to converge beyond $N=50$, exhausting the evaluation budget. In contrast, our ADMM--DFO framework exploits the chain structure by decomposing the decision vector into overlapping blocks of size at most four, with neighboring blocks overlapping in shared variables. Consensus constraints enforce agreement between overlaps, yielding a separable formulation consistent with \eqref{eq:sep_opt_problem}. Algorithmic parameters followed the Arrowhead setup except with an initial penalty parameter $\beta^1=20$, and block subproblems are solved in parallel across the 24 available threads (batched when $N_B>24$). 
The results in Table~\ref{tab:rosen_results} show a similar trend to the earlier case study: for small problem sizes, ADMM--DFO underperforms monolithic methods because coordination costs outweigh the benefits of distribution. As the dimension grows, the method gains efficiency, producing notable reductions in run time with controlled function evaluations. For $N=500$, Figure~\ref{fig:rosen_convergence} shows the as convergence of the augmented Lagrangian and residual norms
$\epsilon_1^{k,r}$, $\epsilon_2^{k,r}$, and $\epsilon_3^{k,r}$ across outer iterations, confirming convergence to a stationary point in this nonconvex setting. These results further demonstrate the scalability of the proposed distributed framework in high-dimensional DFO problems.

% ===== Table: ROSENBROCK results 
\begin{table*}[t]
\centering
\caption{Results for the Rosenbrock function with Nelder--Mead, NEWUOA, and ADMM--DFO solvers. Bold indicates the least computational time and least number of function evaluations (\#Fev) per row. $\dagger$: The maximum number of evaluations was reached.}
\label{tab:rosen_results}

\begingroup
\setlength{\tabcolsep}{2.5pt}
\renewcommand{\arraystretch}{1.0}
\small % <<< change to \footnotesize or \small if preferred

\begin{adjustbox}{max width=\textwidth}
\begin{tabular}{cccccccccccc}
\toprule
\multirow{2}{*}{$N$}
& \multicolumn{4}{c}{Nelder--Mead}
& \multicolumn{4}{c}{NEWUOA}
& \multicolumn{3}{c}{ADMM--DFO} \\
\cmidrule(lr){2-5} \cmidrule(lr){6-9} \cmidrule(lr){10-12}
& Final Obj. & Time (s) & \#Fev & ~
& Final Obj. & Time (s) & \#Fev & ~
& Final Obj. & Time (s) & \#Fev \\
\midrule
10   & $8.86\times10^{-6}$ & 0.06   & 2{,}970 & ~
     & $8.80\times10^{-6}$ & \textbf{0.02} & \textbf{635} & ~
     & $7.75\times10^{-8}$ & 19.53  & 9{,}407 \\
50   & $9.83\times10^{-6}$ & 16.76  & 643{,}991 & ~
     & $9.98\times10^{-6}$ & \textbf{0.77} & \textbf{9{,}701} & ~
     & $4.25\times10^{-7}$ & 100.15 & 28{,}031 \\
100  & $7.80\times10^{1}$  & 34.90  & $\dagger$ & ~
     & $9.88\times10^{-6}$ & \textbf{13.93} & \textbf{35{,}674} & ~
     & $5.50\times10^{-7}$ & 174.18 & 23{,}701 \\
200  & $1.89\times10^{2}$  & 64.31  & $\dagger$ & ~
     & $9.47\times10^{-6}$ & \textbf{247.84} & \textbf{133{,}515} & ~
     & $1.30\times10^{-6}$ & 577.76 & 46{,}763 \\
300  & $2.91\times10^{2}$  & 670.52 & $\dagger$ & ~
     & $9.90\times10^{-6}$ & 1{,}317.36 & 291{,}459 & ~
     & $1.51\times10^{-6}$ & \textbf{996.73} & \textbf{51{,}763} \\
400  & $3.94\times10^{2}$  & 994.71 & $\dagger$ & ~
     & $9.98\times10^{-6}$ & 4{,}394.20 & 507{,}477 & ~
     & $2.39\times10^{-6}$ & \textbf{1{,}517.72} & \textbf{62{,}903} \\
500  & $4.95\times10^{2}$  & 1{,}405.15 & $\dagger$ & ~
     & $9.80\times10^{-6}$ & 20{,}211.24 & 841{,}989 & ~
     & $2.17\times10^{-6}$ & \textbf{2{,}696.52} & \textbf{85{,}719} \\
\bottomrule
\end{tabular}
\end{adjustbox}
\endgroup
\end{table*}

\begin{figure}[htbp] 
\centering
\includegraphics[width=0.6\textwidth]{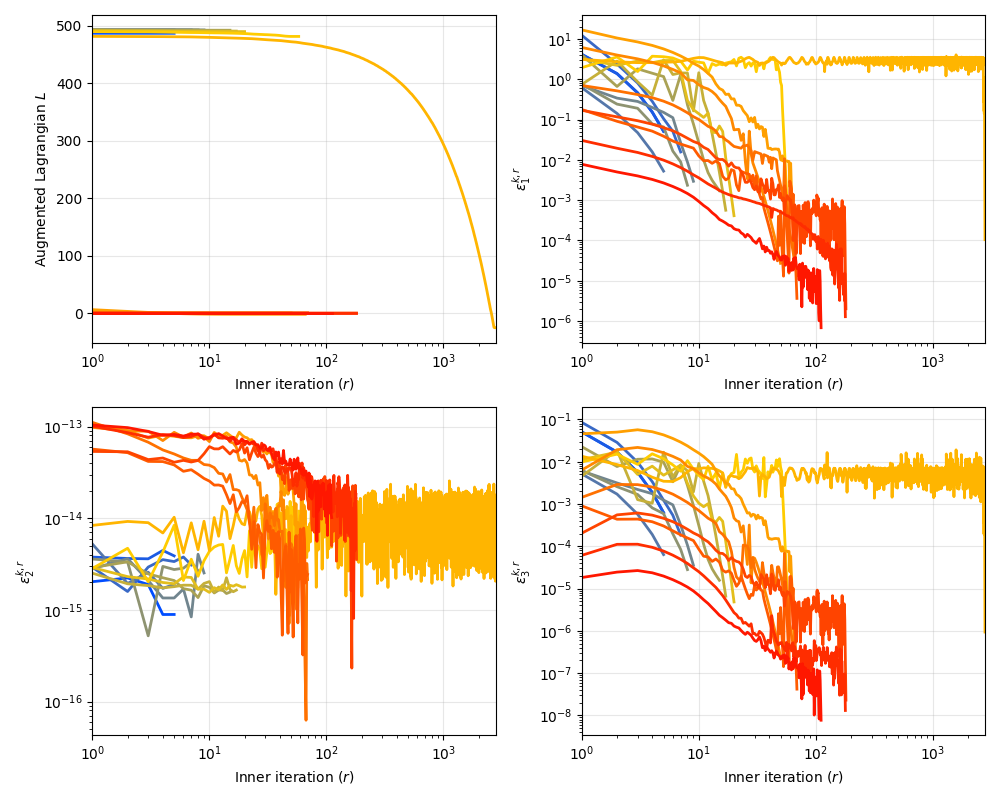}
\caption{Convergence of the augmented Lagrangian and residual norms for the 500-dimensional Rosenbrock problem. Colors indicate progression of outer iterations from blue (initial) to red (final).}
\label{fig:rosen_convergence}
\end{figure}

\subsection{Distributed Nonconvex Nonsmooth Neural Network Training with Consensus}

We now consider a distributed nonconvex and nonsmooth optimization problem arising from multi-agent neural network training under consensus constraints, whose formulation—studied in federated optimization settings in \cite{Qiu2023}—is adapted here. In particular, we solve a problem with \( N = 10 \) agents, each holding a local subset of the Banknote Authentication dataset \cite{Lohweg2012} with \( n_i = 100 \) samples and feature dimension \( 4 \). The global objective is to train a single-hidden-layer ReLU network with \( N_1 = 2 \) hidden units, where all agents must agree on a common weight vector \( \bar{x} \in \mathbb{R}^{M} \), with \( M = N_1(N_0 + 1) = 10 \). The optimization problem is:

\begin{equation}
\min_{x_1,\dots,x_N,\bar{x}} \sum_{i=1}^N f_i(x_i) \quad \text{subject to} \quad x_i = \bar{x}, \quad i = 1,\dots,N,
\end{equation}

\noindent where each local objective \( f_i \) is given by:
\begin{equation}
f_i(x_i) = \frac{1}{2N} \sum_{j=1}^{n_i} \left( y_{ij} - w^\top \max(0, Z^\top u_{ij}) \right)^2 
+ \frac{\lambda}{2N} \left( \|Z\|_F^2 + \|w\|^2 \right).
\end{equation}

Here, \( x_i = [\text{vec}(Z); w] \), \( Z \in \mathbb{R}^{4 \times 2} \) are the hidden-layer weights, \( w \in \mathbb{R}^{2} \) is the output-layer weight vector, \( u_{ij} \in \mathbb{R}^4 \) is the feature vector of the \( j \)-th sample of agent \( i \), and \( y_{ij} \in \{-1, +1\} \) is the corresponding label. The ReLU activation introduces nonsmoothness, while the squared loss with weight decay is nonconvex due to the hidden-layer structure, and since analytic gradients are not assumed available, each \( f_i \) is treated as a black-box function.

The dataset (1372 samples) was randomly split into 1000 training samples (distributed equally across agents) and 372 validation samples. Each agent thus received 100 training samples. Using the same algorithmic parameters as in the previous examples, Figure~\ref{fig:relu_nn_residual} shows the convergence of the augmented Lagrangian and residuals \( \epsilon_1^{k,r} \), \( \epsilon_2^{k,r} \), and \( \epsilon_3^{k,r} \) to zero for \(\beta^1 = 7\), confirming convergence even in this challenging nonsmooth nonconvex learning problem. The solution strategy achieved a validation accuracy of 95.2\% on the held-out set, illustrating the effectiveness of the method.

\begin{figure}[h]
\centering
\includegraphics[width=0.7\textwidth]{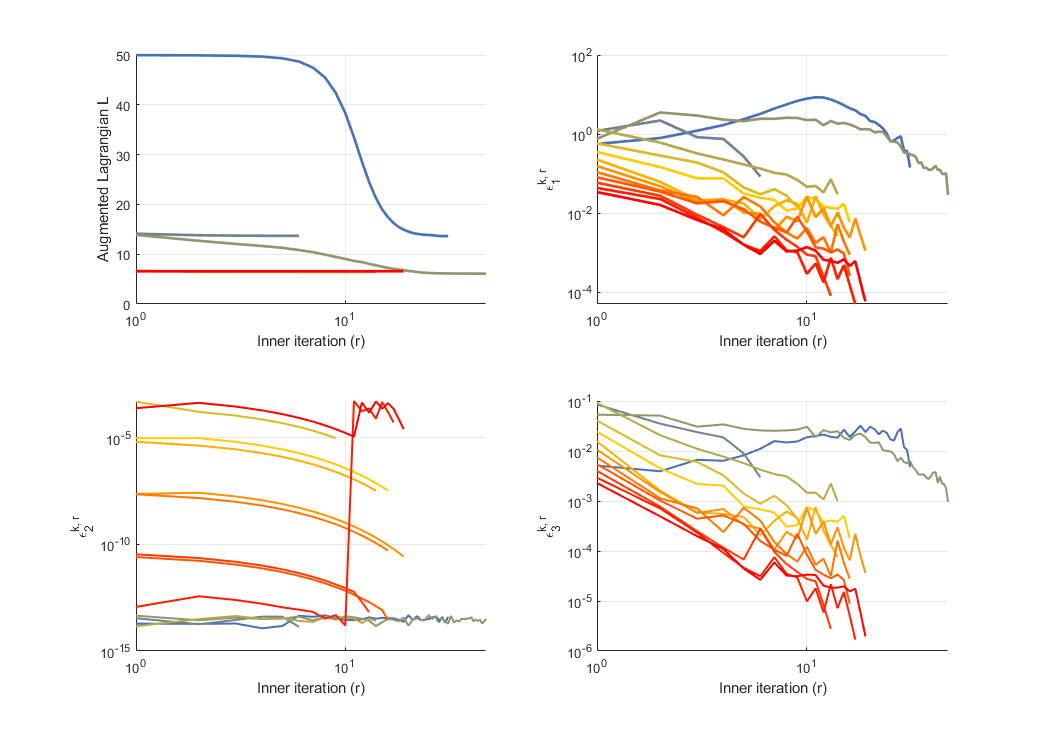}
\caption{Convergence of the augmented Lagrangian and residual norms for the nonconvex nonsmooth problem. Colors indicate progression of outer iterations from blue (initial) to red (final).}
\label{fig:relu_nn_residual}
\end{figure}

\section{Conclusion}\label{sec:conclusion}

This paper has presented a novel distributed algorithm for separable, high-dimensional, nonconvex optimization in derivative-free settings with general linear constraints. The method combines a two-level inexact ADMM framework, which effectively handles complex coupling structures beyond simple consensus problems, with derivative-free trust-region solvers for handling black-box objectives. We have established convergence under mild assumptions in both smooth and nonsmooth cases. In the smooth setting, approximate stationarity points are obtained by exploiting a computable bound on the gradient norm and trust region radius, while in the nonsmooth setting we make an ad hoc assumption for the existence of such a bound, ensuring convergence to approximate Clarke stationarity.  A key innovation is the use of the trust-region radius as a practical measure of stationarity, which provides a systematic mechanism for controlling inexactness while preserving efficiency.

Two categories of problems were highlighted. First, structured high-dimensional DFO problems, where exploiting problem structure enables decomposition into tractable subproblems. Second, distributed consensus optimization in DFO settings, where our framework enforces agreement across agents. Numerical experiments demonstrated that, for challenging large-scale problems, the distributed approach achieves substantial speedups over monolithic solvers. In some instances, problems with more than 200 subproblems were solved efficiently on hardware limited to 24 threads, suggesting even greater gains on high-performance computing platforms. 

Future directions include extending the framework beyond stationary-point convergence to global optimization by integrating data-driven branch-and-bound techniques with the ADMM–DFO solver, constructing convex underestimators from sampled evaluations to provide rigorous lower bounds in derivative-free settings. Additional research directions include deriving explicit convergence-rate guarantees and incorporating surrogate modeling to reduce function evaluations. Together, these extensions will further improve the robustness and scalability of the proposed framework, enabling its application to a wider class of distributed derivative-free optimization problems.

\bmhead{Funding}

This research was supported by the UNC System Research Opportunities Initiative (ROI).

\section*{Declarations}

The authors declare that they have no conflict of interest.

\bibliography{sn-bibliography}


\begin{thebibliography}{54}
\ifx \bisbn   \undefined \def \bisbn  #1{ISBN #1}\fi
\ifx \binits  \undefined \def \binits#1{#1}\fi
\ifx \bauthor  \undefined \def \bauthor#1{#1}\fi
\ifx \batitle  \undefined \def \batitle#1{#1}\fi
\ifx \bjtitle  \undefined \def \bjtitle#1{#1}\fi
\ifx \bvolume  \undefined \def \bvolume#1{\textbf{#1}}\fi
\ifx \byear  \undefined \def \byear#1{#1}\fi
\ifx \bissue  \undefined \def \bissue#1{#1}\fi
\ifx \bfpage  \undefined \def \bfpage#1{#1}\fi
\ifx \blpage  \undefined \def \blpage #1{#1}\fi
\ifx \burl  \undefined \def \burl#1{\textsf{#1}}\fi
\ifx \doiurl  \undefined \def \doiurl#1{\url{https://doi.org/#1}}\fi
\ifx \betal  \undefined \def \betal{\textit{et al.}}\fi
\ifx \binstitute  \undefined \def \binstitute#1{#1}\fi
\ifx \binstitutionaled  \undefined \def \binstitutionaled#1{#1}\fi
\ifx \bctitle  \undefined \def \bctitle#1{#1}\fi
\ifx \beditor  \undefined \def \beditor#1{#1}\fi
\ifx \bpublisher  \undefined \def \bpublisher#1{#1}\fi
\ifx \bbtitle  \undefined \def \bbtitle#1{#1}\fi
\ifx \bedition  \undefined \def \bedition#1{#1}\fi
\ifx \bseriesno  \undefined \def \bseriesno#1{#1}\fi
\ifx \blocation  \undefined \def \blocation#1{#1}\fi
\ifx \bsertitle  \undefined \def \bsertitle#1{#1}\fi
\ifx \bsnm \undefined \def \bsnm#1{#1}\fi
\ifx \bsuffix \undefined \def \bsuffix#1{#1}\fi
\ifx \bparticle \undefined \def \bparticle#1{#1}\fi
\ifx \barticle \undefined \def \barticle#1{#1}\fi
\bibcommenthead
\ifx \bconfdate \undefined \def \bconfdate #1{#1}\fi
\ifx \botherref \undefined \def \botherref #1{#1}\fi
\ifx \url \undefined \def \url#1{\textsf{#1}}\fi
\ifx \bchapter \undefined \def \bchapter#1{#1}\fi
\ifx \bbook \undefined \def \bbook#1{#1}\fi
\ifx \bcomment \undefined \def \bcomment#1{#1}\fi
\ifx \oauthor \undefined \def \oauthor#1{#1}\fi
\ifx \citeauthoryear \undefined \def \citeauthoryear#1{#1}\fi
\ifx \endbibitem  \undefined \def \endbibitem {}\fi
\ifx \bconflocation  \undefined \def \bconflocation#1{#1}\fi
\ifx \arxivurl  \undefined \def \arxivurl#1{\textsf{#1}}\fi
\csname PreBibitemsHook\endcsname

%%% 1
\bibitem[\protect\citeauthoryear{Liang et~al.}{2024}]{liang2024}
\begin{barticle}
\bauthor{\bsnm{Liang}, \binits{R.}},
\bauthor{\bsnm{Han}, \binits{Y.}},
\bauthor{\bsnm{Hu}, \binits{H.}},
\bauthor{\bsnm{Chen}, \binits{B.}},
\bauthor{\bsnm{Yuan}, \binits{Z.}},
\bauthor{\bsnm{Biegler}, \binits{L.T.}}:
\batitle{Efficient trust region filter modeling strategies for computationally expensive black-box optimization}.
\bjtitle{Computers \& Chemical Engineering}
\bvolume{189},
\bfpage{108816}
(\byear{2024})
\end{barticle}
\endbibitem

%%% 2
\bibitem[\protect\citeauthoryear{Conn et~al.}{2009}]{conn2009}
\begin{bbook}
\bauthor{\bsnm{Conn}, \binits{A.R.}},
\bauthor{\bsnm{Scheinberg}, \binits{K.}},
\bauthor{\bsnm{Vicente}, \binits{L.N.}}:
\bbtitle{Introduction to Derivative-Free Optimization}.
\bpublisher{Society for Industrial and Applied Mathematics},
\blocation{Philadelphia}
(\byear{2009})
\end{bbook}
\endbibitem

%%% 3
\bibitem[\protect\citeauthoryear{Boukouvala and Ierapetritou}{2014}]{boukouvala2014}
\begin{barticle}
\bauthor{\bsnm{Boukouvala}, \binits{F.}},
\bauthor{\bsnm{Ierapetritou}, \binits{M.G.}}:
\batitle{Derivative‐free optimization for expensive constrained problems using a novel expected improvement objective function}.
\bjtitle{AIChE Journal}
\bvolume{60}(\bissue{7}),
\bfpage{2462}--\blpage{2474}
(\byear{2014})
\end{barticle}
\endbibitem

%%% 4
\bibitem[\protect\citeauthoryear{Bergstra et~al.}{2011}]{bergstra2011}
\begin{bchapter}
\bauthor{\bsnm{Bergstra}, \binits{J.}},
\bauthor{\bsnm{Bardenet}, \binits{R.}},
\bauthor{\bsnm{Bengio}, \binits{Y.}},
\bauthor{\bsnm{K{\'e}gl}, \binits{B.}}:
\bctitle{Algorithms for hyper-parameter optimization}.
In: \bbtitle{Advances in Neural Information Processing Systems 24 (NeurIPS 2011)}
(\byear{2011})
\end{bchapter}
\endbibitem

%%% 5
\bibitem[\protect\citeauthoryear{Zaryab et~al.}{2022}]{zaryab2022}
\begin{bchapter}
\bauthor{\bsnm{Zaryab}, \binits{S.A.}},
\bauthor{\bsnm{Manno}, \binits{A.}},
\bauthor{\bsnm{Martelli}, \binits{E.}}:
\bctitle{Scr: a novel surrogate-based global optimization algorithm for constrained black-box problems}.
In: \bbtitle{Computer Aided Chemical Engineering}
vol. \bseriesno{51},
pp. \bfpage{1213}--\blpage{1218}.
\bpublisher{Elsevier},
\blocation{Amsterdam}
(\byear{2022})
\end{bchapter}
\endbibitem

%%% 6
\bibitem[\protect\citeauthoryear{Terayama et~al.}{2021}]{terayama2021}
\begin{barticle}
\bauthor{\bsnm{Terayama}, \binits{K.}},
\bauthor{\bsnm{Sumita}, \binits{M.}},
\bauthor{\bsnm{Tamura}, \binits{R.}},
\bauthor{\bsnm{Tsuda}, \binits{K.}}:
\batitle{Black-box optimization for automated discovery}.
\bjtitle{Accounts of Chemical Research}
\bvolume{54}(\bissue{6}),
\bfpage{1334}--\blpage{1346}
(\byear{2021})
\end{barticle}
\endbibitem

%%% 7
\bibitem[\protect\citeauthoryear{Larson et~al.}{2019}]{larson2019}
\begin{barticle}
\bauthor{\bsnm{Larson}, \binits{J.}},
\bauthor{\bsnm{Menickelly}, \binits{M.}},
\bauthor{\bsnm{Wild}, \binits{S.M.}}:
\batitle{Derivative-free optimization methods}.
\bjtitle{Acta Numerica}
\bvolume{28},
\bfpage{287}--\blpage{404}
(\byear{2019})
\end{barticle}
\endbibitem

%%% 8
\bibitem[\protect\citeauthoryear{Powell}{2003}]{powell2003}
\begin{barticle}
\bauthor{\bsnm{Powell}, \binits{M.J.}}:
\batitle{On trust region methods for unconstrained minimization without derivatives}.
\bjtitle{Mathematical Programming}
\bvolume{97},
\bfpage{605}--\blpage{623}
(\byear{2003})
\end{barticle}
\endbibitem

%%% 9
\bibitem[\protect\citeauthoryear{Boukouvala et~al.}{2016}]{boukouvala2016}
\begin{barticle}
\bauthor{\bsnm{Boukouvala}, \binits{F.}},
\bauthor{\bsnm{Misener}, \binits{R.}},
\bauthor{\bsnm{Floudas}, \binits{C.A.}}:
\batitle{Global optimization advances in mixed-integer nonlinear programming, minlp, and constrained derivative-free optimization, cdfo}.
\bjtitle{European Journal of Operational Research}
\bvolume{252}(\bissue{3}),
\bfpage{701}--\blpage{727}
(\byear{2016})
\end{barticle}
\endbibitem

%%% 10
\bibitem[\protect\citeauthoryear{Rios and Sahinidis}{2013}]{rios2013}
\begin{barticle}
\bauthor{\bsnm{Rios}, \binits{L.M.}},
\bauthor{\bsnm{Sahinidis}, \binits{N.V.}}:
\batitle{Derivative-free optimization: a review of algorithms and comparison of software implementations}.
\bjtitle{Journal of Global Optimization}
\bvolume{56}(\bissue{3}),
\bfpage{1247}--\blpage{1293}
(\byear{2013})
\end{barticle}
\endbibitem

%%% 11
\bibitem[\protect\citeauthoryear{Frazier}{2018}]{frazier2018}
\begin{botherref}
\oauthor{\bsnm{Frazier}, \binits{P.I.}}:
A tutorial on {B}ayesian optimization
(2018)
\end{botherref}
\endbibitem

%%% 12
\bibitem[\protect\citeauthoryear{Shahriari et~al.}{2015}]{shahriari2015}
\begin{barticle}
\bauthor{\bsnm{Shahriari}, \binits{B.}},
\bauthor{\bsnm{Swersky}, \binits{K.}},
\bauthor{\bsnm{Wang}, \binits{Z.}},
\bauthor{\bsnm{Adams}, \binits{R.P.}},
\bauthor{\bsnm{De~Freitas}, \binits{N.}}:
\batitle{Taking the human out of the loop: A review of {B}ayesian optimization}.
\bjtitle{Proceedings of the IEEE}
\bvolume{104}(\bissue{1}),
\bfpage{148}--\blpage{175}
(\byear{2015})
\end{barticle}
\endbibitem

%%% 13
\bibitem[\protect\citeauthoryear{Wang and Dowling}{2022}]{wang2022}
\begin{barticle}
\bauthor{\bsnm{Wang}, \binits{K.}},
\bauthor{\bsnm{Dowling}, \binits{A.W.}}:
\batitle{Bayesian optimization for chemical products and functional materials}.
\bjtitle{Current Opinion in Chemical Engineering}
\bvolume{36},
\bfpage{100728}
(\byear{2022})
\end{barticle}
\endbibitem

%%% 14
\bibitem[\protect\citeauthoryear{Lu et~al.}{2020}]{lu2020}
\begin{botherref}
\oauthor{\bsnm{Lu}, \binits{Q.}},
\oauthor{\bsnm{Kumar}, \binits{R.}},
\oauthor{\bsnm{Zavala}, \binits{V.M.}}:
MPC controller tuning using {B}ayesian optimization techniques
(2020)
\end{botherref}
\endbibitem

%%% 15
\bibitem[\protect\citeauthoryear{Kudva et~al.}{2022}]{kudva2022}
\begin{barticle}
\bauthor{\bsnm{Kudva}, \binits{A.}},
\bauthor{\bsnm{Sorourifar}, \binits{F.}},
\bauthor{\bsnm{Paulson}, \binits{J.A.}}:
\batitle{Constrained robust {B}ayesian optimization of expensive noisy black‐box functions with guaranteed regret bounds}.
\bjtitle{AIChE Journal}
\bvolume{68}(\bissue{12}),
\bfpage{17857}
(\byear{2022})
\end{barticle}
\endbibitem

%%% 16
\bibitem[\protect\citeauthoryear{Eason and Biegler}{2016}]{eason2016}
\begin{barticle}
\bauthor{\bsnm{Eason}, \binits{J.P.}},
\bauthor{\bsnm{Biegler}, \binits{L.T.}}:
\batitle{A trust region filter method for glass box/black box optimization}.
\bjtitle{AIChE Journal}
\bvolume{62}(\bissue{9}),
\bfpage{3124}--\blpage{3136}
(\byear{2016})
\end{barticle}
\endbibitem

%%% 17
\bibitem[\protect\citeauthoryear{Bajaj et~al.}{2018}]{bajaj2018}
\begin{barticle}
\bauthor{\bsnm{Bajaj}, \binits{I.}},
\bauthor{\bsnm{Iyer}, \binits{S.S.}},
\bauthor{\bsnm{Hasan}, \binits{M.F.}}:
\batitle{A trust region-based two phase algorithm for constrained black-box and grey-box optimization with infeasible initial point}.
\bjtitle{Computers \& Chemical Engineering}
\bvolume{116},
\bfpage{306}--\blpage{321}
(\byear{2018})
\end{barticle}
\endbibitem

%%% 18
\bibitem[\protect\citeauthoryear{Conn et~al.}{2000}]{conn2000}
\begin{bbook}
\bauthor{\bsnm{Conn}, \binits{A.R.}},
\bauthor{\bsnm{Gould}, \binits{N.I.M.}},
\bauthor{\bsnm{Toint}, \binits{P.L.}}:
\bbtitle{Trust-Region Methods}.
\bpublisher{Society for Industrial and Applied Mathematics},
\blocation{Philadelphia}
(\byear{2000})
\end{bbook}
\endbibitem

%%% 19
\bibitem[\protect\citeauthoryear{Conn et~al.}{2009}]{conn2009global}
\begin{barticle}
\bauthor{\bsnm{Conn}, \binits{A.R.}},
\bauthor{\bsnm{Scheinberg}, \binits{K.}},
\bauthor{\bsnm{Vicente}, \binits{L.N.}}:
\batitle{Global convergence of general derivative-free trust-region algorithms to first- and second-order critical points}.
\bjtitle{SIAM Journal on Optimization}
\bvolume{20}(\bissue{1}),
\bfpage{387}--\blpage{415}
(\byear{2009})
\end{barticle}
\endbibitem

%%% 20
\bibitem[\protect\citeauthoryear{Shashaani et~al.}{2018}]{shashaani2018}
\begin{barticle}
\bauthor{\bsnm{Shashaani}, \binits{S.}},
\bauthor{\bsnm{Hashemi}, \binits{F.S.}},
\bauthor{\bsnm{Pasupathy}, \binits{R.}}:
\batitle{Astro-df: A class of adaptive sampling trust-region algorithms for derivative-free stochastic optimization}.
\bjtitle{SIAM Journal on Optimization}
\bvolume{28}(\bissue{4}),
\bfpage{3145}--\blpage{3176}
(\byear{2018})
\end{barticle}
\endbibitem

%%% 21
\bibitem[\protect\citeauthoryear{Ghanbari and Scheinberg}{2017}]{ghanbari2017}
\begin{botherref}
\oauthor{\bsnm{Ghanbari}, \binits{H.}},
\oauthor{\bsnm{Scheinberg}, \binits{K.}}:
Black-box optimization in machine learning with trust region based derivative free algorithm
(2017)
\end{botherref}
\endbibitem

%%% 22
\bibitem[\protect\citeauthoryear{Garmanjani et~al.}{2016}]{garmanjani2016}
\begin{barticle}
\bauthor{\bsnm{Garmanjani}, \binits{R.}},
\bauthor{\bsnm{J{\'u}dice}, \binits{D.}},
\bauthor{\bsnm{Vicente}, \binits{L.N.}}:
\batitle{Trust-region methods without using derivatives: worst case complexity and the nonsmooth case}.
\bjtitle{SIAM Journal on Optimization}
\bvolume{26}(\bissue{4}),
\bfpage{1987}--\blpage{2011}
(\byear{2016})
\end{barticle}
\endbibitem

%%% 23
\bibitem[\protect\citeauthoryear{Liuzzi et~al.}{2019}]{Liuzzi2019}
\begin{barticle}
\bauthor{\bsnm{Liuzzi}, \binits{G.}},
\bauthor{\bsnm{Lucidi}, \binits{S.}},
\bauthor{\bsnm{Rinaldi}, \binits{F.}},
\bauthor{\bsnm{Vicente}, \binits{L.N.}}:
\batitle{Trust-region methods for the derivative-free optimization of nonsmooth black-box functions}.
\bjtitle{SIAM Journal on Optimization}
\bvolume{29}(\bissue{4}),
\bfpage{3012}--\blpage{3035}
(\byear{2019})
\end{barticle}
\endbibitem

%%% 24
\bibitem[\protect\citeauthoryear{Ling et~al.}{2017}]{ling2017}
\begin{barticle}
\bauthor{\bsnm{Ling}, \binits{J.}},
\bauthor{\bsnm{Hutchinson}, \binits{M.}},
\bauthor{\bsnm{Antono}, \binits{E.}},
\bauthor{\bsnm{Paradiso}, \binits{S.}},
\bauthor{\bsnm{Meredig}, \binits{B.}}:
\batitle{High-dimensional materials and process optimization using data-driven experimental design with well-calibrated uncertainty estimates}.
\bjtitle{Integrating Materials and Manufacturing Innovation}
\bvolume{6},
\bfpage{207}--\blpage{217}
(\byear{2017})
\end{barticle}
\endbibitem

%%% 25
\bibitem[\protect\citeauthoryear{He et~al.}{2021}]{he2021}
\begin{barticle}
\bauthor{\bsnm{He}, \binits{J.}},
\bauthor{\bsnm{You}, \binits{H.}},
\bauthor{\bsnm{Sandstr{\"o}m}, \binits{E.}},
\bauthor{\bsnm{Nittinger}, \binits{E.}},
\bauthor{\bsnm{Bjerrum}, \binits{E.J.}},
\bauthor{\bsnm{Tyrchan}, \binits{C.}},
\bauthor{\bsnm{Engkvist}, \binits{O.}}:
\batitle{Molecular optimization by capturing chemist’s intuition using deep neural networks}.
\bjtitle{Journal of Cheminformatics}
\bvolume{13},
\bfpage{1}--\blpage{17}
(\byear{2021})
\end{barticle}
\endbibitem

%%% 26
\bibitem[\protect\citeauthoryear{Qian et~al.}{2016}]{qian2016}
\begin{bchapter}
\bauthor{\bsnm{Qian}, \binits{H.}},
\bauthor{\bsnm{Hu}, \binits{Y.Q.}},
\bauthor{\bsnm{Yu}, \binits{Y.}}:
\bctitle{Derivative-free optimization of high-dimensional non-convex functions by sequential random embeddings}.
In: \bbtitle{Proceedings of the International Joint Conference on Artificial Intelligence (IJCAI)},
pp. \bfpage{1946}--\blpage{1952}
(\byear{2016})
\end{bchapter}
\endbibitem

%%% 27
\bibitem[\protect\citeauthoryear{Wang et~al.}{2016}]{wang2016}
\begin{barticle}
\bauthor{\bsnm{Wang}, \binits{Z.}},
\bauthor{\bsnm{Hutter}, \binits{F.}},
\bauthor{\bsnm{Zoghi}, \binits{M.}},
\bauthor{\bsnm{Matheson}, \binits{D.}},
\bauthor{\bsnm{De~Freitas}, \binits{N.}}:
\batitle{Bayesian optimization in a billion dimensions via random embeddings}.
\bjtitle{Journal of Artificial Intelligence Research}
\bvolume{55},
\bfpage{361}--\blpage{387}
(\byear{2016})
\end{barticle}
\endbibitem

%%% 28
\bibitem[\protect\citeauthoryear{Song et~al.}{2022}]{song2022}
\begin{bchapter}
\bauthor{\bsnm{Song}, \binits{L.}},
\bauthor{\bsnm{Xue}, \binits{K.}},
\bauthor{\bsnm{Huang}, \binits{X.}},
\bauthor{\bsnm{Qian}, \binits{C.}}:
\bctitle{Monte carlo tree search based variable selection for high dimensional bayesian optimization}.
In: \bbtitle{Advances in Neural Information Processing Systems},
vol. \bseriesno{35},
pp. \bfpage{28488}--\blpage{28501}
(\byear{2022})
\end{bchapter}
\endbibitem

%%% 29
\bibitem[\protect\citeauthoryear{Rolland et~al.}{2018}]{rolland2018}
\begin{bchapter}
\bauthor{\bsnm{Rolland}, \binits{P.}},
\bauthor{\bsnm{Scarlett}, \binits{J.}},
\bauthor{\bsnm{Bogunovic}, \binits{I.}},
\bauthor{\bsnm{Cevher}, \binits{V.}}:
\bctitle{High-dimensional {B}ayesian optimization via additive models with overlapping groups}.
In: \bbtitle{Proceedings of the International Conference on Artificial Intelligence and Statistics (AISTATS)},
pp. \bfpage{298}--\blpage{307}.
\bpublisher{PMLR},
\blocation{Playa Blanca, Lanzarote, Spain}
(\byear{2018})
\end{bchapter}
\endbibitem

%%% 30
\bibitem[\protect\citeauthoryear{Porcelli and Toint}{2022}]{porcelli2022}
\begin{barticle}
\bauthor{\bsnm{Porcelli}, \binits{M.}},
\bauthor{\bsnm{Toint}, \binits{P.L.}}:
\batitle{Exploiting problem structure in derivative-free optimization}.
\bjtitle{ACM Transactions on Mathematical Software (TOMS)}
\bvolume{48}(\bissue{1}),
\bfpage{1}--\blpage{25}
(\byear{2022})
\end{barticle}
\endbibitem

%%% 31
\bibitem[\protect\citeauthoryear{Dantzig and Wolfe}{1960}]{dantzig1960}
\begin{barticle}
\bauthor{\bsnm{Dantzig}, \binits{G.B.}},
\bauthor{\bsnm{Wolfe}, \binits{P.}}:
\batitle{Decomposition principle for linear programs}.
\bjtitle{Operations Research}
\bvolume{8}(\bissue{1}),
\bfpage{101}--\blpage{111}
(\byear{1960})
\end{barticle}
\endbibitem

%%% 32
\bibitem[\protect\citeauthoryear{Benders}{1962}]{benders2005}
\begin{barticle}
\bauthor{\bsnm{Benders}, \binits{J.F.}}:
\batitle{Partitioning procedures for solving mixed-variables programming problems}.
\bjtitle{Numerische Mathematik}
\bvolume{4},
\bfpage{238}--\blpage{252}
(\byear{1962})
\end{barticle}
\endbibitem

%%% 33
\bibitem[\protect\citeauthoryear{Boyd et~al.}{2011}]{boyd2011}
\begin{barticle}
\bauthor{\bsnm{Boyd}, \binits{S.}},
\bauthor{\bsnm{Parikh}, \binits{N.}},
\bauthor{\bsnm{Chu}, \binits{E.}},
\bauthor{\bsnm{Peleato}, \binits{B.}},
\bauthor{\bsnm{Eckstein}, \binits{J.}}:
\batitle{Distributed optimization and statistical learning via the alternating direction method of multipliers}.
\bjtitle{Foundations and Trends in Machine Learning}
\bvolume{3}(\bissue{1}),
\bfpage{1}--\blpage{122}
(\byear{2011})
\end{barticle}
\endbibitem

%%% 34
\bibitem[\protect\citeauthoryear{Glowinski and Marroco}{1975}]{glowinski1975}
\begin{barticle}
\bauthor{\bsnm{Glowinski}, \binits{R.}},
\bauthor{\bsnm{Marroco}, \binits{A.}}:
\batitle{Sur l'approximation, par {\'e}l{\'e}ments finis d'ordre un, et la r{\'e}solution, par p{\'e}nalisation-dualit{\'e} d'une classe de probl{\`e}mes de {D}irichlet non lin{\'e}aires}.
\bjtitle{Revue fran{\c c}aise d'automatique, informatique, recherche op{\'e}rationnelle. Analyse num{\'e}rique}
\bvolume{9}(\bissue{R2}),
\bfpage{41}--\blpage{76}
(\byear{1975})
\end{barticle}
\endbibitem

%%% 35
\bibitem[\protect\citeauthoryear{Gabay and Mercier}{1976}]{gabay1976}
\begin{barticle}
\bauthor{\bsnm{Gabay}, \binits{D.}},
\bauthor{\bsnm{Mercier}, \binits{B.}}:
\batitle{A dual algorithm for the solution of nonlinear variational problems via finite element approximation}.
\bjtitle{Computers \& Mathematics with Applications}
\bvolume{2}(\bissue{1}),
\bfpage{17}--\blpage{40}
(\byear{1976})
\end{barticle}
\endbibitem

%%% 36
\bibitem[\protect\citeauthoryear{Ryu and Yin}{2022}]{ryu2022}
\begin{bbook}
\bauthor{\bsnm{Ryu}, \binits{E.K.}},
\bauthor{\bsnm{Yin}, \binits{W.}}:
\bbtitle{Large-Scale Convex Optimization: Algorithms and Analyses Via Monotone Operators}.
\bpublisher{Cambridge University Press},
\blocation{Cambridge}
(\byear{2022})
\end{bbook}
\endbibitem

%%% 37
\bibitem[\protect\citeauthoryear{Spingarn}{1985}]{spingarn1985}
\begin{barticle}
\bauthor{\bsnm{Spingarn}, \binits{J.E.}}:
\batitle{Applications of the method of partial inverses to convex programming: Decomposition}.
\bjtitle{Mathematical Programming}
\bvolume{32}(\bissue{2}),
\bfpage{199}--\blpage{223}
(\byear{1985})
\end{barticle}
\endbibitem

%%% 38
\bibitem[\protect\citeauthoryear{Krishnamoorthy and Paulson}{2023}]{krishnamoorthy2023}
\begin{barticle}
\bauthor{\bsnm{Krishnamoorthy}, \binits{D.}},
\bauthor{\bsnm{Paulson}, \binits{J.A.}}:
\batitle{Multi-agent black-box optimization using a {B}ayesian approach to alternating direction method of multipliers}.
\bjtitle{IFAC-PapersOnLine}
\bvolume{56}(\bissue{2}),
\bfpage{2232}--\blpage{2237}
(\byear{2023})
\end{barticle}
\endbibitem

%%% 39
\bibitem[\protect\citeauthoryear{Krishnamoorthy}{2025}]{krishnamoorthy2025}
\begin{barticle}
\bauthor{\bsnm{Krishnamoorthy}, \binits{D.}}:
\batitle{Eccbo: An inherently safe {B}ayesian optimization with embedded constraint control for real-time process optimization}.
\bjtitle{Journal of Process Control}
\bvolume{152},
\bfpage{103467}
(\byear{2025})
\end{barticle}
\endbibitem

%%% 40
\bibitem[\protect\citeauthoryear{Rodriguez et~al.}{2018}]{rodriguez2018}
\begin{barticle}
\bauthor{\bsnm{Rodriguez}, \binits{J.S.}},
\bauthor{\bsnm{Nicholson}, \binits{B.}},
\bauthor{\bsnm{Laird}, \binits{C.}},
\bauthor{\bsnm{Zavala}, \binits{V.M.}}:
\batitle{Benchmarking {ADMM} in nonconvex {NLP}s}.
\bjtitle{Computers \& Chemical Engineering}
\bvolume{119},
\bfpage{315}--\blpage{325}
(\byear{2018})
\end{barticle}
\endbibitem

%%% 41
\bibitem[\protect\citeauthoryear{Hong et~al.}{2016}]{hong2016}
\begin{barticle}
\bauthor{\bsnm{Hong}, \binits{M.}},
\bauthor{\bsnm{Luo}, \binits{Z.-Q.}},
\bauthor{\bsnm{Razaviyayn}, \binits{M.}}:
\batitle{Convergence analysis of alternating direction method of multipliers for a family of nonconvex problems}.
\bjtitle{SIAM Journal on Optimization}
\bvolume{26}(\bissue{1}),
\bfpage{337}--\blpage{364}
(\byear{2016})
\end{barticle}
\endbibitem

%%% 42
\bibitem[\protect\citeauthoryear{Wang et~al.}{2019}]{wang2019}
\begin{barticle}
\bauthor{\bsnm{Wang}, \binits{Y.}},
\bauthor{\bsnm{Yin}, \binits{W.}},
\bauthor{\bsnm{Zeng}, \binits{J.}}:
\batitle{Global convergence of {ADMM} in nonconvex nonsmooth optimization}.
\bjtitle{Journal of Scientific Computing}
\bvolume{78},
\bfpage{29}--\blpage{63}
(\byear{2019})
\end{barticle}
\endbibitem

%%% 43
\bibitem[\protect\citeauthoryear{Tang and Daoutidis}{2022}]{tang2022}
\begin{barticle}
\bauthor{\bsnm{Tang}, \binits{W.}},
\bauthor{\bsnm{Daoutidis}, \binits{P.}}:
\batitle{Fast and stable nonconvex constrained distributed optimization: the {ELLADA} algorithm}.
\bjtitle{Optimization and Engineering}
\bvolume{23}(\bissue{1}),
\bfpage{259}--\blpage{301}
(\byear{2022})
\end{barticle}
\endbibitem

%%% 44
\bibitem[\protect\citeauthoryear{Hager and Zhang}{2019}]{hager2019}
\begin{barticle}
\bauthor{\bsnm{Hager}, \binits{W.W.}},
\bauthor{\bsnm{Zhang}, \binits{H.}}:
\batitle{Inexact alternating direction methods of multipliers for separable convex optimization}.
\bjtitle{Computational Optimization and Applications}
\bvolume{73},
\bfpage{201}--\blpage{235}
(\byear{2019})
\end{barticle}
\endbibitem

%%% 45
\bibitem[\protect\citeauthoryear{Bai et~al.}{2025}]{bai2025}
\begin{barticle}
\bauthor{\bsnm{Bai}, \binits{J.}},
\bauthor{\bsnm{Zhang}, \binits{M.}},
\bauthor{\bsnm{Zhang}, \binits{H.}}:
\batitle{An inexact {ADMM} for separable nonconvex and nonsmooth optimization}.
\bjtitle{Computational Optimization and Applications}
\bvolume{90}(\bissue{2}),
\bfpage{445}--\blpage{479}
(\byear{2025})
\end{barticle}
\endbibitem

%%% 46
\bibitem[\protect\citeauthoryear{Eckstein and Yao}{2017}]{eckstein2017}
\begin{barticle}
\bauthor{\bsnm{Eckstein}, \binits{J.}},
\bauthor{\bsnm{Yao}, \binits{W.}}:
\batitle{Approximate {ADMM} algorithms derived from {L}agrangian splitting}.
\bjtitle{Computational Optimization and Applications}
\bvolume{68},
\bfpage{363}--\blpage{405}
(\byear{2017})
\end{barticle}
\endbibitem

%%% 47
\bibitem[\protect\citeauthoryear{Xie et~al.}{2017}]{xie2017}
\begin{barticle}
\bauthor{\bsnm{Xie}, \binits{J.}},
\bauthor{\bsnm{Liao}, \binits{A.}},
\bauthor{\bsnm{Yang}, \binits{X.}}:
\batitle{An inexact alternating direction method of multipliers with relative error criteria}.
\bjtitle{Optimization Letters}
\bvolume{11},
\bfpage{583}--\blpage{596}
(\byear{2017})
\end{barticle}
\endbibitem

%%% 48
\bibitem[\protect\citeauthoryear{Sun and Sun}{2023}]{sun2023}
\begin{barticle}
\bauthor{\bsnm{Sun}, \binits{K.}},
\bauthor{\bsnm{Sun}, \binits{X.A.}}:
\batitle{A two-level distributed algorithm for nonconvex constrained optimization}.
\bjtitle{Computational Optimization and Applications}
\bvolume{84}(\bissue{2}),
\bfpage{609}--\blpage{649}
(\byear{2023})
\end{barticle}
\endbibitem

%%% 49
\bibitem[\protect\citeauthoryear{Clarke et~al.}{1998}]{Clarke1998}
\begin{bbook}
\bauthor{\bsnm{Clarke}, \binits{F.H.}},
\bauthor{\bsnm{Ledyaev}, \binits{Y.S.}},
\bauthor{\bsnm{Stern}, \binits{R.J.}},
\bauthor{\bsnm{Wolenski}, \binits{R.R.}}:
\bbtitle{Nonsmooth Analysis and Control Theory}.
\bpublisher{Springer},
\blocation{New York}
(\byear{1998})
\end{bbook}
\endbibitem

%%% 50
\bibitem[\protect\citeauthoryear{Rockafellar and Wets}{1998}]{RockafellarWets1998}
\begin{bbook}
\bauthor{\bsnm{Rockafellar}, \binits{R.T.}},
\bauthor{\bsnm{Wets}, \binits{R.J.-B.}}:
\bbtitle{Variational Analysis}.
\bsertitle{Grundlehren der Mathematischen Wissenschaften},
vol. \bseriesno{317}.
\bpublisher{Springer},
\blocation{Berlin}
(\byear{1998})
\end{bbook}
\endbibitem

%%% 51
\bibitem[\protect\citeauthoryear{Conn et~al.}{1994}]{Conn1994}
\begin{bchapter}
\bauthor{\bsnm{Conn}, \binits{A.R.}},
\bauthor{\bsnm{Gould}, \binits{N.}},
\bauthor{\bsnm{Lescrenier}, \binits{M.}},
\bauthor{\bsnm{Toint}, \binits{P.L.}}:
\bctitle{Performance of a multifrontal scheme for partially separable optimization}.
In: \bbtitle{Advances in Optimization and Numerical Analysis},
pp. \bfpage{79}--\blpage{96}.
\bpublisher{Springer},
\blocation{Dordrecht}
(\byear{1994})
\end{bchapter}
\endbibitem

%%% 52
\bibitem[\protect\citeauthoryear{Powell}{2006}]{Powell2006}
\begin{bchapter}
\bauthor{\bsnm{Powell}, \binits{M.J.D.}}:
\bctitle{The {NEWUOA} software for unconstrained optimization without derivatives}.
In: \beditor{\bsnm{Di~Pillo}, \binits{G.}},
\beditor{\bsnm{Roma}, \binits{M.}} (eds.)
\bbtitle{Large-Scale Nonlinear Optimization},
pp. \bfpage{255}--\blpage{297}.
\bpublisher{Springer},
\blocation{New York}
(\byear{2006})
\end{bchapter}
\endbibitem

%%% 53
\bibitem[\protect\citeauthoryear{Qiu et~al.}{2023}]{Qiu2023}
\begin{bchapter}
\bauthor{\bsnm{Qiu}, \binits{Y.}},
\bauthor{\bsnm{Shanbhag}, \binits{U.}},
\bauthor{\bsnm{Yousefian}, \binits{F.}}:
\bctitle{Zeroth-order methods for nondifferentiable, nonconvex, and hierarchical federated optimization}.
In: \bbtitle{Advances in Neural Information Processing Systems},
vol. \bseriesno{36},
pp. \bfpage{3425}--\blpage{3438}
(\byear{2023})
\end{bchapter}
\endbibitem

%%% 54
\bibitem[\protect\citeauthoryear{Lohweg}{2012}]{Lohweg2012}
\begin{botherref}
\oauthor{\bsnm{Lohweg}, \binits{V.}}:
Banknote Authentication [Dataset].
UCI Machine Learning Repository.
\url{https://doi.org/10.24432/C55P57}
(2012)
\end{botherref}
\endbibitem

\end{thebibliography}

\end{document}